\documentclass[a4paper,10pt,reqno]{amsart}
\usepackage[utf8]{inputenc}
\usepackage{mathtools}
\usepackage{amssymb}
\usepackage{amsfonts}
\usepackage{amsthm}
\usepackage{bbm}
\usepackage{fullpage}
\usepackage{color}
\usepackage{hyperref}
\usepackage{stmaryrd}
\usepackage{enumitem}
\usepackage{amsaddr}

\newcommand{\N}{\mathbb N}
\newcommand{\Z}{\mathbb Z}
\newcommand{\R}{\mathbb R}
\newcommand{\Q}{\mathbb Q}
\newcommand{\eps}{\varepsilon}
\renewcommand{\P}{\mathbb P}
\newcommand{\E}{\mathbb E}
\newcommand{\1}{\mathbf 1}
\newcommand{\0}{\mathbf 0}
\newcommand{\F}{\mathcal F}
\newcommand{\ind}{\mathbbm 1}

\newtheorem{theorem}{Theorem}[section]
\newtheorem{lemma}[theorem]{Lemma}
\newtheorem{proposition}[theorem]{Proposition}

\newtheorem{thmx}{Theorem}

\theoremstyle{remark}
\newtheorem{example}[theorem]{Example}
\newtheorem{remark}[theorem]{Remark}

\theoremstyle{definition}
\newtheorem{definition}[theorem]{Definition}

\newcommand{\dd}{\text{d}}

\newcommand{\h}{h}
\newcommand{\esssup}{\operatorname{ess\,sup}}
\newcommand{\essinf}{\operatorname{ess\,inf}}

\newcommand{\aalpha}{{\boldsymbol{\alpha}}}
\newcommand{\ggamma}{{\boldsymbol{\alpha}}}

\renewcommand{\S}{\mathcal S}

\newcommand{\U}{\mathcal U}

\newcommand{\rrho}{{\boldsymbol{\rho}}}

\newcommand{\p}{\mathfrak p}

\newcommand{\n}{\llbracket n\rrbracket}

\newcommand{\M}{\mathcal M}
\newcommand{\eeta}{{\boldsymbol{\eta}}}

\newcommand{\Omegas}{\Omega^{\,\textup{site}}}
\newcommand{\Ps}{\P^{\,\textup{site}}}

\newcommand{\Ts}{T^{\,\textup{site}}}

\newcommand{\Omegab}{\Omega^{\textup{bond}}}
\newcommand{\Pb}{\P^{\textup{bond}}}

\newcommand{\Tb}{T^{\textup{bond}}}

\newcommand{\Omegap}{\Omega^{\textup{Pois.}}}
\newcommand{\Pp}{\P^{\textup{Pois.}}}

\newcommand{\PSRW}{P^{\textup{SRW}}}
\newcommand{\PBM}{P^{\textup{BM}}}

\newcommand{\llambda}{\boldsymbol{\lambda}}
\newcommand{\mmu}{\boldsymbol{z}}

\makeatletter
\def\namedlabel#1#2#3{\begingroup
    #2%
    \def\@currentlabel{#3}%
    \phantomsection\label{#1}\endgroup
}
\makeatother

\newcommand{\nnu}{\mu}

\title[Stability of weak disorder phase for directed polymer]{Stability of weak disorder phase for directed polymer with applications to limit theorems}

\author{Stefan Junk}
\address{AIMR, Tohoku University. 2-1-1 Katahira, Aoba-ku, Sendai, 980-8577 Japan}
\email{sjunk@tohoku.ac.jp}
\date{\today}

\subjclass[2010]{60K37}
\keywords{Random dynamics, random environments, scaling limit}

\begin{document}

\lineskip=0pt

\begin{abstract}
We study the directed polymer model in a bounded environment with bond disorder and show that, in the interior of the weak disorder phase, weak disorder continues to hold upon perturbation by a small bias. Using this stability result, we give a new proof for the central limit theorem (CLT) in probability for the directed polymer model in the interior of the weak disorder phase. We also show that the large deviation rate function agrees with that of the underlying random walk. For the Brownian polymer model, we improve the convergence in the CLT to almost sure convergence in the whole weak disorder phase. The main technical tools are a new moment bound from \cite{J21_1} and a quantitative comparison between the associated martingales at different inverse temperatures.
\end{abstract}

\maketitle

\section{New tools}\label{sec:new}

Before introducing the model and the results, we briefly discuss the new ideas introduced in this paper.

\smallskip The \emph{directed polymer model} was initially introduced in the physics literature to describe the folding of long molecule chains in a solution with random impurities. Mathematically, it is a model for random paths, called \emph{polymers}, that are attracted or repulsed by a space-time random environment with a parameter $\beta\geq 0$, called \emph{inverse temperature}, governing the strength of the interaction. In recent years, the model has attracted much interest because it is conjectured that in a certain \emph{low temperature} regime it belongs to the KPZ (Kardar-Parisi-Zhang) universality class of randomly growing surfaces. In contrast, this work is mainly concerned with the \emph{high temperature} phase, where it is known that the influence of the disorder disappears asymptotically and that the long-term behavior is diffusive. This weak disorder phase is characterized by whether a certain associated martingale, $(W_n^\beta)_{n\in\N}$, is uniformly integrable, which is known to hold for small $\beta$ if the spatial dimension is at least three.

\smallskip There is no closed-form characterization for the critical inverse temperature $\beta_{cr}$ and, in practice, the uniform integrability is a fairly weak integrability condition that is not easy to analyze. Several important features of the weak disorder phase have been established, see for example \cite{CY03,CY06,BC20}, but many more papers have focused on a different, \emph{very high temperature} phase, which is characterized by $L^2$-boundedness of the associated martingale. This condition is known to be strictly stronger than uniform integrability, cf. \cite[Remark 5.2(ii)]{C17} for the references. In the recent work \cite{J21_1}, we managed to improve this situation by showing that, in the whole weak disorder phase and for a certain class of environments, $(W_n^\beta)_{n\in\N}$ is $L^p$-bounded for some $p>1$. This result plays an important role in the current paper by allowing us to interchange limits and expectations. 

\smallskip This work introduces a second new tool for the study of the weak disorder phase. Recall from \cite{N16} that larger values of $\beta$ make the martingale $W_n^\beta$ ``more variable'', in the sense that $\beta\mapsto \E[f(W_n^\beta)]$ is increasing for every $f\colon\R\to\R$ convex. While this monotonicity agrees with the intuition that $\beta$ is a measure for the strength of disorder in the model, it does not provide any quantitative information. Our new idea is to use the so-called \emph{noise operator} $T_\rho$ as a measure for how variability increases with $\beta$. Here, $T_\rho$ acts on the environment by retaining each coordinate with probability $\rho$ and resampling it with probability $1-\rho$, independently. This operator has featured prominently in problems motived by computer science, social choice and combinatorics, see \cite{O14,GS15}, but has, to the best of our knowledge, not appeared in the context of the directed polymer model before. We show that for every convex function $f\colon\R\to\R$,
\begin{align}\label{eq:noise}
\E[f(W_n^{\beta_1})]\leq \E[f(T_\rho W_n^{\beta_2})]
\end{align}
where $0<\beta_1<\beta_2$ and where $\rho<1$ has an explicit expression in terms of $\beta_1$ and $\beta_2$. Roughly speaking, \eqref{eq:noise} allows us to integrate out parts of the environment at the ``cost'' of increasing the inverse temperature.

\smallskip This result is particularly useful if we consider \emph{bond disorder}, in contrast to the \emph{site disorder} studied most often in the literature. To give an example, in Proposition~\ref{prop:drift} we apply \eqref{eq:noise} to integrate out the disorder on a proportion of the edges pointing in, say, the $e_1$-direction and show that the resulting object can be interpreted as the martingale associated  to a polymer with a drift in the opposite direction. We note that, somewhat unexpectedly, our argument only works for bond disorder -- in general, bond and site disorder are regarded as equivalent and the focus on site disorder in the literature seems to be based mainly on notational convenience and convention.

\section{Model and results}

\subsection{Definition of the model}\label{sec:defs}

We study three variations of the directed polymer model: two of them are in discrete time with space-time disorder associated to either the sites or the bonds of $\N\times\Z^d$, and the third model, called the \emph{Brownian polymer model}, is in continuous time with the disorder given by a Poisson point process. Let $(\omega_0,\P_0)$ be a probability measure on $(\Omega_0,\F_0)=(\R,\mathcal B(\R))$, where we assume %throughout that
\begin{align}\label{eq:exp_mom}
\E\left[e^{\beta|\omega_0|}\right]<\infty\qquad\text{ for all }\beta\geq 0.
\end{align}
We also assume that $\omega_0$ is not almost surely constant. The notation is similar for all three models:
\begin{description}
 \item[\namedlabel{itm:site}{Site disorder (site dis.)}{site dis.}] Let $\Omegas=\Omega_0^{\N\times\Z^d}$, $\omega=(\omega_{t,x})_{t\in\N,x\in\Z^d}$ and $\Ps$ the product measure, $\Ps=\bigotimes_{(t,x)\in\N\times\Z^d}\P_0$. 

 \label{it:site}
 \item[\namedlabel{itm:bond}{Bond disorder (bond dis.)}{bond dis.}] Let $\Omegab=\Omega_0^{E}$, $\omega=(\omega_e)_{e\in E}$ and $\Pb=\bigotimes_{e\in E}\P_0$, where $E$ denotes the set of directed nearest-neighbor edges in $\N\times\Z^d$,
 \begin{align}
E&\coloneqq \Big\{\big((n,x),(n+1,x+u)\big)\colon n\in\N,x\in\Z^d,u\in\U\Big\},\label{eq:edges}\\
\U&\coloneqq \{\pm e_i\colon i=1,\dots,d\}.\label{eq:def_U}
\end{align}
\item[\namedlabel{itm:pois}{Poissonian disorder (pois. dis.)}{pois. dis.}] Let $\Omegap$ be the set of locally finite counting measures $\omega$ on $\R_+\times\R^d\times\Omega_0$ and $\Pp$ the Poisson point process with intensity measure $\dd t\,\dd x\,\P_0(\dd \eta)$. We identify $\omega\in\Omegap$ with its support, writing $(t,x,\eta)\in\omega$ in place of $\omega(\{(t,x,\eta)\})=1$.
\end{description}

The unperturbed model is a simple random walk $(X=(X_n)_{n\in\N},\PSRW)$ for the discrete-time models and Brownian motion $(B=(B_t)_{t\geq 0},\PBM)$ for \eqref{itm:pois}. The perturbation is defined as a Gibbs measure, where the \emph{energy} of a path in environment $\omega$ is given by
\begin{align*}
H_n(\omega,x)&\textstyle \coloneqq \sum_{i=1}^n\omega_{i,x_i}\qquad\text{ for }\eqref{itm:site},\\
H_n(\omega,x)&\textstyle \coloneqq \sum_{i=1}^n \omega_{((i-1,x_{i-1}),(i,x_i))}\qquad\text{ for }\eqref{itm:bond}, \\
H_t(\omega,b)&\coloneqq \textstyle\sum_{(s,y,\eta)\in \omega}\eta\ind\{s\in[0,t],y\in U(b_s)\}\qquad\text{ for }\eqref{itm:pois}.
\end{align*}
Here $U(x)$ denotes the ball of unit volume around $x\in\R^d$. In words, the energy of a path is the sum, resp. integral, of the environment observed along the path. The \emph{polymer measure} is defined by
\begin{equation}\label{eq:poly_meas}
\begin{split}
\mu_{\omega,n}^\beta(\dd X)&\coloneqq  (Z_n^\beta(\omega))^{-1}e^{\beta H_n(\omega,X)} \PSRW(\dd X)\qquad\text{ for }\eqref{itm:site}\text{ and }\eqref{itm:bond},\\
 \mu_{\omega,t}^\beta(\dd B)&\coloneqq (Z_t^\beta(\omega))^{-1}e^{\beta H_t(\omega,B)}\PBM(\dd B)\qquad\text{ for }\eqref{itm:pois},
\end{split}
\end{equation}
where $Z_t^\beta$ is the normalizing constant, called the \emph{partition function} of the model. Under the polymer measure, paths are attracted to areas of space-time where the environment is positive and repelled by areas where it is negative. Note that $(\mu_{\omega,n}^\beta)_{n\in\N}$ is not a consistent family, i.e., there is no infinite volume probability measure ``$\mu_{\omega,\infty}^\beta$'' whose projection to time $n$ agrees with $\mu_{\omega,n}^\beta$, simultaneously for all $n\in\N$. The \emph{associated martingale} mentioned in the previous section is defined by 
\begin{align*}
W_n^\beta(\omega)&\coloneqq  Z_n^{\beta}(\omega)e^{-n\lambda(\beta)}\qquad\text{ for }\eqref{itm:site}\text{ and }\eqref{itm:bond},\\
W_t^\beta(\omega)&\coloneqq  Z_t^{\beta}(\omega)e^{-t(e^{\lambda(\beta)}-1)}\qquad\text{ for }\eqref{itm:pois},
\end{align*}
where $\lambda(\beta)\coloneqq \log \E[e^{\beta\omega_0}]$. 

\smallskip We also introduce notation for adding a drift to the underlying simple random walk, resp. Brownian motion, which we need for the perturbative arguments used to prove the main results. We use bold symbols to indicate vector-valued parameters. First, for the discrete-time models, let $\M(\U)$ denote the set of probability measures on $\U$ (recall \eqref{eq:def_U}). For $\aalpha\in\M(\U)$, we write $P^\aalpha$ for the random walk with increment distribution $\aalpha$, i.e.,
\begin{align}\label{eq:RW_drift}
P^\aalpha(X_{n+1}-X_n=u)=\aalpha_u,
\end{align}
and $Z_n^{\beta,\aalpha}$, $W_n^{\beta,\aalpha}$ and $\mu_{\omega,n}^{\beta,\aalpha}$ for the partition function, associated martingale and polymer measure with $\PSRW$ replaced by $P^\aalpha$. Similarly, for the Brownian polymer model, let $\llambda\in\R^d$ and write $P^{\llambda}$ for the law of Brownian motion with drift $\llambda$, 
\begin{align*}
P^{\llambda}(\dd B)=e^{\llambda\cdot B_t-\frac t2\|\llambda\|_2^2}\PBM(\dd B),
\end{align*}
and define $Z_t^{\beta,\llambda}$, $W_t^{\beta,\llambda}$ and $\mu_{\omega,t}^{\beta,\llambda}$ as in the discrete-time case. If no drift is specified, we always refer to the symmetric case $\aalpha=\0$, resp. $\llambda=\0$.

As a non-negative martingale, the limit $W_\infty^{\beta,\aalpha}\coloneqq\lim_{t\to\infty}W_t^{\beta,\aalpha}\geq 0$ exists almost surely. It characterizes the phase transition from high to low temperature, i.e., we say that \emph{weak disorder} \eqref{eq:WD} holds if
\begin{align}\label{eq:WD}\tag{WD}
\P(W_\infty^{\beta,\aalpha}>0)>0
\end{align}
and that \emph{strong disorder} \eqref{eq:SD} holds otherwise, i.e., if
\begin{align}\label{eq:SD}\tag{SD}
\P(W_\infty^{\beta,\aalpha}=0)=1.
\end{align}
It is not hard to see that $W_\infty^\beta$ satisfies a zero-one law, i.e., for all $\beta\geq 0$,
\begin{align}\label{eq:zero_one}
\P\left(W_\infty^\beta>0\right)\in\{0,1\}.
\end{align}
Finally, we say that the environment is \emph{upper bounded} \eqref{eq:upper_bd}, \emph{lower bounded} \eqref{eq:lower_bd} or \emph{bounded} \eqref{eq:bd} if there exists $K>0$ such that, respectively,
\begin{align}
\P_0\left([K,\infty)\right)&=0,\tag{U-Bd.}\label{eq:upper_bd}\\
\P_0((-\infty,-K])&=0,\tag{L-Bd.}\label{eq:lower_bd}\\
\P_0((-\infty,-K]\cup[K,\infty))&=0.\tag{Bd.}\label{eq:bd}
\end{align}

\subsection{Known results}\label{sec:known}

We review some results about the model to put our results into context. A more complete overview can be found in \cite{CSY04} and \cite{C17} for \eqref{itm:site} and \cite{CC18} for \eqref{itm:pois}. %No result will be used in the proofs.

\smallskip A major question in statistical mechanics is about ``disorder relevance'', i.e., whether the long-term behavior of a (randomly) perturbed model is different from the unperturbed model. In the present context, it is natural to ask whether the sequence $(\mu_{\omega,n}^\beta)_{n\in\N}$ satisfies a central limit theorem similar to the infinite temperature model $(\mu_{\omega,n}^0)_{n\in\N}$.
The answer depends on whether weak disorder \eqref{eq:WD} or strong disorder \eqref{eq:SD} hold, as we will recall in the following. First, we have a phase transition in $\beta$ and $L^2$-boundedness holds deep inside the weak disorder phase.

\begin{thmx}\label{thmx:phase}
Consider any of the three models introduced above, and let $T=\N$ for the discrete-time models \eqref{itm:site} and \eqref{itm:bond} and $T=\R_+$ for \eqref{itm:pois}.
\begin{enumerate}
 \item[(i)] There exist critical values $\beta_{cr},\beta_{cr}^{L^2}\in[0,\infty)$ such that 
 \begin{itemize}
  \item \eqref{eq:WD} holds for $\beta<\beta_{cr}$ and \eqref{eq:SD} for $\beta>\beta_{cr}$
  \item $(W_t^\beta)_{t\in T}$ is $L^2$-bounded if and only if $\beta<\beta_{cr}^{L^2}$.
 \end{itemize}
 \item[(ii)] If $d\geq 3$, then $0<\beta_{cr}^{L^2}\leq \beta_{cr}$. Moreover, for \eqref{itm:site}, $\beta_{cr}^{L^2}<\beta_{cr}$.
\end{enumerate}
\end{thmx}
See \cite[Theorem 1.1]{CY03}, \cite[Theorem 2.1.1]{CY05} and the references in \cite[Remark 5.2(ii)]{C17} in the case of \ref{itm:site}. We did not find references for \ref{itm:bond}, but the proofs from \ref{itm:site} apply with minimal modifications. Also, while the strict inequality $\beta_{cr}^{L^2}<\beta_{cr}$ in part (ii) has only been obtained for \ref{itm:site}, there is no reason to expect that the situation is different for the other models. Whether the critical value $\beta_{cr}$ belongs to the strong or weak disorder phase is a major open problem, see \cite[Open Problem 3.2]{C17}.

\smallskip As mentioned in Section~\ref{sec:new}, stronger integrability has recently been proved for bounded environments in the whole weak disorder phase.

\begin{thmx}[{\cite[Theorem 1.1]{J21_1}}]\label{thmx:sj}
Consider \eqref{itm:site} or \eqref{itm:bond}. If \eqref{eq:WD} holds, then
\begin{align}\label{eq:sup_integrable}
\E\left[\sup_{n\in\N}W_n^{\beta,\aalpha}\right]<\infty.
\end{align}
If additionally \eqref{eq:upper_bd} holds, then there exists $p>1$ such that
\begin{align}\label{eq:Lp_bd}
\sup_{n\in\N}\|W_n^{\beta,\aalpha}\|_p<\infty.
\end{align}
Analogous results hold for \eqref{itm:pois}.
\end{thmx}

In \cite{J21_1}, the proof of Theorem~\ref{thmx:sj} is given for site disorder in the symmetric case, but the extensions to general drifts, to bond disorder and to continuous space-time require only minor modifications. The question at the beginning of this section about the long-term behavior of $(\mu_{\omega,n}^\beta)_{n\in\N}$ has been answered for the whole weak disorder phase:

\begin{thmx}[{\cite[Theorem 1.2]{CY06}}]\label{thmx:clt}
Consider \eqref{itm:site} in dimension $d\geq 3$ and assume \eqref{eq:WD}. Then for all $f\colon \R^d\to\R$ bounded and continuous,
\begin{align}\label{eq:clt}
\sum_{x\in\Z^d}\mu_{\omega,n}^\beta(X_n=x)f\left(\frac{x}{\sqrt n}\right)\xrightarrow[n\to\infty]P \int_{\R^d}f(x) k(x)\dd x,
\end{align}
where $k$ is the density of the $d$-dimensional standard normal distribution.  
\end{thmx}

Equivalent results are known for \eqref{itm:pois}, see 
\cite[Proposition A.2]{CNN20}. We will give an independent proof of \eqref{eq:clt} for bond disorder in the interior of the weak disorder phase $\beta<\beta_{cr}$, see Theorem~\ref{thm:clt_prob} below. Moreover, we show that for the Brownian polymer model, convergence in probability can be improved to almost sure convergence, see Theorem~\ref{thm:clt_as_pois}.

\smallskip The fact that the central limit theorem continues to hold for some $\beta>0$ came as a surprise to the community, and in this spirit the following result is interesting:

\begin{thmx}[{\cite[Theorem 6.2]{CY06}}]\label{thmx:analytic}
Consider \eqref{itm:site} in dimension $d\geq 3$ and let $\beta<\beta_{cr}^{L^2}$. Then, almost surely, $\lim_{n\to\infty}\frac{1}{n}\mu_{\omega,n}^\beta[H_n(\omega,X)]=\lambda'(\beta)$.
\end{thmx}

To put this result into context, note that for $\beta>0$ the set of paths $\pi$ with $H_n(\omega,\pi)\approx n\lambda'(\beta)$ is exponentially small in the set of all paths, so that $\mu_{\omega,n}^\beta$ is supported on a negligibly small part of the support of $P$. Thus \eqref{eq:clt} holds on the diffusive scale despite the fact that paths under $\mu_{\omega,n}^\beta$ look locally quite different from paths under the unperturbed measure $P$.

\smallskip To close this discussion, we mention that much stronger results are available in the $L^2$-phase: in addition to Theorem~\ref{thmx:analytic}, it is known that convergence in probability in \eqref{eq:clt} can be replaced by almost sure convergence \cite{IS88,B89,SZ96}; precise error bounds for the rate of convergence to $W_\infty^\beta$ have been established \cite{CL17,CN20}; 
and a local central limit theorem has been proved \cite{S95,V06}. We also mention that much research has focused on the strong disorder phase \eqref{eq:SD}. In particular, the one-dimensional case is a very active field of research because of its conjectured relation to the KPZ universality class and because exactly solvable models are known. The behavior of $\mu_{\omega,n}^\beta$ in strong disorder is radically different from the weak disorder phase that is the focus of the present article.

\subsection{The main results}
\subsubsection{Central limit theorem for bond disorder}

We give a new proof for \eqref{eq:clt} in the interior of the weak disorder phase for bond disorder. Here, we consider $\mu_{\omega,n}^\beta$ as a random variable in $\M(\R^d)$, the set of probability measures on $\R^d$, equipped with the topology of weak convergence, which is induced by the Prokhorov metric $d_P$ on $\M(\R^d)$. As indicated above, our argument uses that weak disorder is stable under perturbation by a small drift, which is our first main result.

\begin{theorem}\label{thm:clt_prob}
Consider \eqref{itm:bond} in dimension $d\geq 3$, assume that the environment is bounded \eqref{eq:bd} and let $\beta<\beta_{cr}$. 
\begin{enumerate}
 \item[(i)] There exists $\alpha_0(\beta)>0$ such that weak disorder holds for all $\aalpha\in(\1/2d+[-\alpha_0,\alpha_0]^d)\cap\M(\U)$, i.e.,
 \begin{align*}
\P(W_\infty^{\beta,\aalpha}>0)=1\quad\text{ for all }\aalpha\in\big(\1/2d+[-\alpha_0,\alpha_0]^d\big)\cap\M(\U).
\end{align*}
\item[(ii)] Let $\mathcal N$ denote the standard normal distribution on $\R^d$. It holds that
\begin{align}
d_P\Big(\mu_{\omega,n}^\beta\Big(\frac{X_n}{\sqrt n}\in\cdot \Big),\mathcal N\Big)\xrightarrow[n\to\infty]P0.
\end{align}
\end{enumerate}
\end{theorem}

While not stated in the literature, the proof of the CLT for site disorder Theorem~\ref{thmx:clt} can presumably be adapted to bond disorder. Theorem~\ref{thm:clt_prob} is potentially weaker because it only covers the interior of the weak disorder phase. However, it is widely expected \cite[Open Problem 3.2]{C17} that strong disorder holds at the critical temperature $\beta_{cr}$, i.e., that the weak disorder phase is open. We think that the argument can be extended to prove a full invariance principle in probability, see Remark~\ref{rm:invariance}, but since the result is not new we do not try to prove the most general statement.

\subsubsection{Central limit theorem for Poissonian disorder}\label{sec:as}

For the Brownian polymer model, there is no need to use perturbative arguments as in Theorem~\ref{thm:clt_prob} due to  a certain shift invariance, which we now explain. Namely, for $\omega\in\Omegap$ and $\llambda\in\R^d$, we define $\omega(\llambda)\in\Omegap$ by
\begin{align}\label{eq:omega_alpha}
(t,x,\eta)\in\omega\qquad\iff\qquad (t,x-t\llambda,\eta)\in\omega(\llambda).
\end{align}
Then $\omega(\llambda)$ has the same law as $\omega$ and one can check that $W_t^{\beta,\llambda}(\omega)=W_t^{\beta}(\omega(\llambda))$. In particular, the law of $W_t^{\beta,\llambda}$ does not depend on $\llambda$  and if \eqref{eq:WD} holds for one $\llambda$, then $\P(W_\infty^{\beta,\llambda'}>0)=1$ for \emph{all} $\llambda'\in\R^d$.

\smallskip We now use this shift invariance and the moment bound from Theorem~\ref{thmx:sj} to improve the convergence in probability from Theorem~\ref{thmx:clt} to almost sure  convergence.

\begin{theorem}\label{thm:clt_as_pois}
Consider \eqref{itm:pois}, assume that weak disorder \eqref{eq:WD} holds and that the environment is upper bounded \eqref{eq:upper_bd}. Let $\mathcal N$ denote the standard normal distribution on $\R^d$. Then, almost surely,
\begin{align*}
d_P\Big(\mu_{\omega,t}^{\beta}\Big(\frac{X_t}{\sqrt t}\in\cdot\Big),\mathcal N\Big)\xrightarrow{t\to\infty}0.
\end{align*}
\end{theorem}

Our proof does not work for the bond disorder model because shift invariance does not hold, but we nevertheless obtain a certain ``defective'' version of Theorem~\ref{thm:clt_as_pois}. To state it, we define $\aalpha(\llambda)\in\M(\U)$ by
\begin{align}\label{eq:alpha_lambda}
(\aalpha(\llambda))_u=\frac{e^{\llambda_u}}{\sum_{v\in\U}e^{\llambda_{v}}}.
\end{align}

\begin{theorem}\label{thm:clt_as_bond}
Consider \eqref{itm:bond} and assume that the environment is bounded \eqref{eq:bd}. Let $\beta<\beta_{cr}$ and let $\lambda_0>0$ be small enough that $\aalpha(\llambda)\in \1/2d+[-\alpha_0,\alpha_0]^d$ for all $\llambda\in[-\lambda_0,\lambda_0]^d$, where $\alpha_0$ is as defined in Theorem~\ref{thm:clt_prob}~(i). For $\P$-almost all $\omega$ there exists a Borel set $\Lambda(\omega)\subseteq [-\lambda_0,\lambda_0]^d$ with $|\Lambda(\omega)|=(2\lambda_0)^d$ such that, for all $\llambda\in \Lambda(\omega)$,
\begin{align*}
d_P\Big(\mu_{\omega,n}^{\beta,\aalpha(\llambda)}\Big(\frac{X_n-\aalpha(\llambda) n}{\sqrt n}\in\cdot\Big),\mathcal N\Big)\xrightarrow{n\to\infty}0.
\end{align*}
\end{theorem}
That is, an almost sure central limit theorem holds for Lebesgue almost all directions $\aalpha\in(\1/2d+[-\alpha_0,\alpha_0]^d)\cap\M(\U)$. Of course, this does not exclude the possibility that $\P\left(\0\notin \Lambda(\omega)\right)>0$ and hence we cannot say anything about the symmetric case $\llambda=\0$. In the case of \eqref{itm:pois}, we use the shift invariance to exclude the possibility that $\0$ is atypical.

\subsubsection{Large deviation principle for bond disorder}\label{sec:ldp}

We present a second application for the perturbative argument used to prove Theorem~\ref{thm:clt_prob}. To motivate it, we briefly discuss the so-called \emph{curvature conjecture}. First, it is known that the polymer endpoint satisfies a large deviation principle (LDP).

\begin{thmx}[{\cite[Theorem 1.2]{CH04},\cite[Theorem 9.1]{C17}}]\label{thmx:free}
Consider \eqref{itm:site} or \eqref{itm:bond} in arbitrary dimension and for any $\beta\geq 0$, $\aalpha\in\M(\U)$. 
\begin{enumerate}
 \item[(i)] For any $x\in\R^d$ with $\|x\|_1\leq 1$, let $(x_n)_{n\in\N}$ be a sequence in $\Z^d$ such that $\lim_{n\to\infty}\frac{x_n}n=x$ and $P^\aalpha(X_n=x_n)>0$. There exist $\p(\beta,\aalpha),\p(\beta,x)\in(-\infty,0]$ such that, almost surely,
 \begin{align}
\p(\beta,\aalpha)&=\lim_{n\to\infty}\frac 1n\log W_n^{\beta,\aalpha} =\lim_{n\to\infty}\frac 1n\E[\log W_n^{\beta,\aalpha}],\label{eq:free}\\
\p(\beta,x)&=\lim_{n\to\infty}\frac 1n \log E[e^{\beta H_n(\omega,X)-n\lambda(\beta)}\ind_{X_n=x_n}] \notag\\&=\lim_{n\to\infty}\frac 1n \E\left[\log E[e^{\beta H_n(\omega,X)-n\lambda(\beta)}\ind_{X_n=x_n}]\right]\label{eq:free_p2p}.
\end{align}
Moreover, $\p(\beta,x)$ does not depend on the choice of $(x_n)_{n\in\N}$.
\item[(ii)] The sequence $(\mu_{\omega,n}^\beta(X_n/n\in\cdot ))_{n\geq 0}$ almost surely satisfies an LDP with deterministic, convex and continuous rate function
\begin{align*}
J^\beta(x)\coloneqq \p(\beta,\1/2d)-\p(\beta,x).
\end{align*}
\end{enumerate}
\end{thmx}

Similar to our convention $W_n^{\beta}=W_n^{\beta,\1/2d}$, we will write $\p(\beta)$ in place of $\p(\beta,\1/2d)$. The quantity in \eqref{eq:free}, resp. \eqref{eq:free_p2p}, is known as the point-to-plane, resp. point-to-point, \emph{free energy}. It is commonly defined with $W_n^{\beta,\aalpha}$ replaced by the partition function $Z_n^{\beta,\aalpha}$, but this only leads a shift by $\lambda(\beta)$. It seems that Theorem~\ref{thmx:free} is only available for site disorder in the symmetric case $\aalpha=\1/2d$ in the literature, but the necessary modifications are straightforward by following the arguments from \cite{CH04,C17}.

\smallskip The curvature conjecture \cite[Open Problem 9.3]{C17} states that $x\mapsto J^\beta(x)$ is strictly convex, and in particular that $J^\beta(x)>0$ for all $x\neq 0$. In strong disorder, the conjecture is related to certain fluctuation exponents, which are expected to follow the KPZ scaling relation, and it would resolve long-standing questions about the long-term behavior of the polymer. For example, without assuming strict convexity of $J^\beta$ it is currently not even known whether a law of large numbers holds  in strong disorder, i.e., whether 
\begin{align*}
\lim_{n\to\infty}\mu_{\omega,n}^\beta\big(|X_n|>\eps n\big)=0.
\end{align*}
In this section, we show that in the interior of the weak disorder phase $J^\beta$ agrees with the rate function of the unperturbed model $\PSRW$ in some neighborhood of the origin. This result is of more limited interest than the curvature conjecture in strong disorder, because both the scaling exponents and the law of large numbers are already known, but we think that the ideas are potentially relevant in strong disorder as well. Namely, in Proposition~\ref{prop:drift_free} below we prove a comparison between the free energies $\p(\beta,\aalpha)$ and $\p(\beta,\aalpha')$ with different drifts $\aalpha,\aalpha'\in\M(\U)$, which should be relevant because an LDP is typically proved by exponentially tilting the original measure. An interesting observation is that in this context bond disorder seems to be easier to analyze than site disorder. 

\smallskip The function $\beta\mapsto \p(\beta)$ is known to be continuous and decreasing, so there exists $\overline\beta_{cr}$ such that $\p(\beta)=0$ if and only if $\beta\leq\overline{\beta}_{cr}$. Clearly $\beta_{cr}\leq \overline{\beta}_{cr}$, and it is believed that the two critical values are in fact equal \cite[Open Problem 3.2]{C17}. We now identify the LDP rate function $J^\beta$ for $\beta<\overline\beta_{cr}$.

\begin{theorem}\label{thm:ldp}
Consider \eqref{itm:bond} in dimension $d\geq 3$ and assume that the environment is bounded \eqref{eq:bd}. Then the following hold.
\begin{enumerate}
 \item[(i)] Let $\beta<\overline\beta_{cr}$. There exists $\alpha_0>0$ such that $\p(\beta,\aalpha)=0$ for all $\aalpha\in(\1/2d+[-\alpha_0,\alpha_0]^d)\cap\M(\U)$.
 \item[(ii)] Let $\beta<\overline\beta_{cr}$. The rate function $J^\beta$ agrees with the rate function $I^{\textup{SRW}}$ of $\PSRW$ in a neighborhood of the origin.
 \item[(iii)] For all $\beta\leq\overline\beta_{cr}$ and all $x\in\R^d$, $J^{\beta}(x)\geq I(x)$.
\end{enumerate}
\end{theorem}

\begin{remark}
\begin{enumerate}
 \item[(i)] We think that Theorem~\ref{thm:ldp} is sharp in the sense that $J^\beta(x)>I^{\textup{SRW}}(x)$ for all $\beta>\beta_{cr}$ and $x\neq 0$. Note that this would imply the curvature conjecture discussed above.
 \item[(ii)] It is not hard to see that 
 \begin{align*}
J^\beta(e_1)=\p(\beta)-\beta\E[\omega_0]+\lambda(\beta)+I^{\textup{SRW}}(e_1),
\end{align*}
 so by continuity we have $J^\beta>I^{\textup{SRW}}$ in some neighborhood of $e_1$. In other words, the disorder irrelevance close to the origin from Theorem~\ref{thm:ldp}~(ii) does not extend to the whole domain.
\end{enumerate}
\end{remark}

The result is inspired by \cite[Exercise 9.1]{C17}, where the same conclusion is obtained for $\beta<\beta_{cr}^{L^2}$ using the so-called replica trick. Recently \cite{FJ21}, we proved an analogue of Theorem~\ref{thm:ldp} for a related model in continuous time and discrete space, which is parametrized in terms of the jump rate $\kappa$ of the underlying random walk instead of the inverse temperature $\beta$. In place of Theorem~\ref{thm:semigroup}, we used the comparison result \cite[Theorem~1]{J20} to ``integrate out'' the drift at the ``cost'' of decreasing the jump rate $\kappa$, but this idea is not easily adapted to discrete time: indeed, the set of discrete-time nearest-neighbor paths is not closed under coordinate-wise addition, so $\PSRW$ does not enjoy the convolution property discussed in the beginning of \cite[Section 2]{FJ21}. Moreover, the inverse temperature $\beta$ is not a direct analogue of the jump rate $\kappa$.

\subsection{Auxiliary results}

In the next section, we introduce the new tool mentioned in the introduction. Afterwards, in Section~\ref{sec:drift}, we state some technical results necessary for the perturbative arguments in Theorems~\ref{thm:clt_prob} and \ref{thm:ldp}. We only consider the discrete-time models \eqref{itm:site} and \eqref{itm:bond}. A suitable version of Theorem~\ref{thm:semigroup} should, in principle, be valid in the continuous-time setting  as well, but due to the shift invariance discussed at the beginning of Section~\ref{sec:as} the generalization of Proposition~\ref{prop:drift} is not interesting for the Brownian polymer model. 

\subsubsection{A quantitative comparison result}\label{sec:results_semigroup}

\smallskip We start by introducing the \emph{noise operator} $
\Ts_\rho$ for site disorder. Let $\widetilde\omega$ be an independent copy of $\omega$ and $(U_{t,x})_{t\in\N,x\in\Z^d}$ an independent family of uniformly distributed random variables. For $\rho\in[0,1]$, define $\Ts_\rho\colon L^2(\Omegas)\to L^2(\Omegas)$ by
\begin{align}\label{eq:T_rho_def}
(\Ts_\rho f)(\omega)\coloneqq \E\left[f(\omega_\rho)\big|\omega\right],
\end{align}
where $\omega_\rho\in\Omegas$ is given by
\begin{align}\label{eq:omega_rho}
\omega_\rho(t,x)\coloneqq 
\begin{cases}
\omega(t,x)&\text{ if }U_{t,x}\leq\rho\\
\widetilde\omega(t,x)&\text{ if }U_{t,x}>\rho.
\end{cases}
\end{align}
In words, $\Ts_\rho$ acts on $\omega$ by resampling each coordinate with probability $1-\rho$, independently. A useful generalization is to consider a family $\rrho=(\rrho_{t,x})_{t\in\N,x\in\Z^d}$ of parameters and define the \emph{inhomogeneous noise operator} by $(\Ts_\rrho f)(\omega)\coloneqq \E\left[f(\omega_\rrho)\big|\omega\right]$, where
\begin{align}\label{eq:omega_rho_inho}
\omega_\rrho(t,x)\coloneqq 
\begin{cases}
\omega(t,x)&\text{ if }U_{t,x}\leq\rrho_{t,x}\\
\widetilde\omega(t,x)&\text{ if }U_{t,x}>\rrho_{t,x}.
\end{cases}
\end{align}
The noise operator $\Tb_\rho\colon L^2(\Omegab)\to L^2(\Omegab)$ for \eqref{itm:bond} is defined similarly by independently resampling the environment at each bond. We now state the comparison result mentioned in Section~\ref{sec:new}.

\begin{theorem}\label{thm:semigroup}
Consider either \eqref{itm:bond} or \eqref{itm:site} and assume that the environment is bounded \eqref{eq:bd}. Let $0<\beta_-<\beta_+<\infty$.
\begin{enumerate}
\item[(i)] For every $\aalpha\in\M(\U)$ and $\beta_1,\beta_2\in[\beta_-,\beta_+]$ with $\beta_1\leq\beta_2$, there exists $\rho_0^{\beta_-,\beta_+}(\beta_1,\beta_2)\in[0,1]$ such that, for every $\rrho\geq \rho_0^{\beta_-,\beta_+}(\beta_1,\beta_2)$, $n\in\N$ and $f\colon \R_+\to\R$ convex,
\begin{align}\label{eq:semigroup}
\E\left[f(W_n^{\beta_1,\aalpha})\right]\leq \E\left[f(T_\rrho W_n^{\beta_2,\aalpha})\right].
\end{align}
\item[(ii)] There exists $C=C(\beta_-,\beta_+)$ such that, for every $\beta_1,\beta_2\in[\beta_-,\beta_+]$ with $\beta_1\leq\beta_2$,  \begin{align}\label{eq:rho_estimate}
\rho_0^{\beta_-,\beta_+}(\beta_1,\beta_2)\leq 1-C(\beta_2/\beta_1-1).
\end{align}
\item[(iii)] For every $\aalpha\in\M(\U)$ and $\beta_1,\beta_2\in[\beta_-,\beta_+]$ with $\beta_1\leq\beta_2$ and $\rrho\geq \rho_0^{\beta_-,\beta_+}(\beta_1,\beta_2)$, there exists a coupling $(\omega_1,\omega_2)$ with marginals $\P$ such that, for every $n\in\N$,
\begin{align}\label{eq:semigroup_coupling}
W_n^{\beta_1,\aalpha}(\omega_1)=\E\left[(T_\rrho W_n^{\beta_2,\aalpha})(\omega_2)\big|\omega_1\right].
\end{align}
\end{enumerate}
\end{theorem}

\begin{remark}
\begin{enumerate}
 \item[(i)] Note that $T_{1}W_n^{\beta}=W_n^\beta$, so \eqref{eq:semigroup} in particular implies that $\beta\mapsto \E[f(W_n^\beta)]$ is increasing in $[\beta_-,\beta_+]$, partially recovering the result from \cite[Theorem 1.2]{N16}.
 \item[(ii)] The assumption of boundedness is sharp. More precisely, without \eqref{eq:bd} the only $\rho$ satisfying \eqref{eq:semigroup} is $\rho_0=1$, so there is no improvement on the result from  \cite{N16}. See Remark~\ref{rem:lorenz} below.
 \item[(iii)] It is known that \eqref{eq:semigroup} and \eqref{eq:semigroup_coupling} are equivalent for any fixed $n$, see \cite[Theorem 3.A.4]{SS07}. The value of Theorem~\ref{thm:semigroup}(iii) is that there is a coupling satisfying \eqref{eq:semigroup_coupling} simultaneously for all $n\in\N$, which we need in Proposition~\ref{prop:drift}(ii).
 \item[(iv)] The constant $C$ from Theorem~\ref{thm:semigroup}(ii) is explicit and depends only on the marginal $\P_0$ of $\P$, see \eqref{eq:rho0}.
 \item[(v)] This result is motivated by \cite[5.A.7.b]{majorization}, where a similar statement appeared for random variables with finite support, although they did not investigate under which conditions $\rho_0<1$ holds.
\end{enumerate}
\end{remark}

An illustrative example is the Bernoulli environment, where we can choose $\omega_1=\omega_2$ in \eqref{eq:semigroup_coupling}.

\begin{example}
Assume $\P_0(\omega_0=0)=1-\P(\omega_0=1)=p\in(0,1)$
. We claim that there exists $\rho=\rho(\beta_1,\beta_2)$ such that almost surely
\begin{align*}
W_n^{\beta_1}(\omega)=(T_\rho W_n^{\beta_2})(\omega).
\end{align*}
Indeed, requiring $e^{\beta_1\omega-\lambda(\beta_1)}=T_\rho e^{\beta_2\omega-\lambda(\beta_2)}$ for both $\omega=0$ and $\omega=1$ results in two equations,
\begin{align*}
e^{-\lambda(\beta_1)}&=\rho e^{-\lambda(\beta_2)}+(1-\rho),\\
e^{\beta_1-\lambda(\beta_1)}&=\rho e^{\beta_2-\lambda(\beta_2)}+(1-\rho).
\end{align*}
Some calculations reveal that both are solved by $\rho\coloneqq \frac{e^{\lambda(\beta_2)}(e^{\beta_1}-1)}{e^{\lambda(\beta_1)}(e^{\beta_2}-1)}$.
\end{example}

\subsubsection{Stability under perturbation by small drifts for bond disorder}\label{sec:drift}

Our main conclusion from Theorem~\ref{thm:semigroup} is that for bond disorder, \eqref{eq:WD} is stable under perturbation by small drifts, which is the content of Proposition~\ref{prop:drift} below. To make this precise, we introduce a ``distance'' between $\aalpha,\aalpha'\in\M(\U)$,
\begin{align}
m(\aalpha|\aalpha')&\coloneqq \min\Big\{\frac{\aalpha_u}{\aalpha_u'}:u\in\U\Big\},\label{eq:def_m}\\
d(\aalpha,\aalpha')&\coloneqq m(\aalpha|\aalpha')m(\aalpha'|\aalpha)\label{eq:def_d}.
\end{align}
Indeed, note that $d(\aalpha,\aalpha'),m(\aalpha|\aalpha')\in[0,1]$ and $d(\aalpha,\aalpha')=m(\aalpha|\aalpha')=1$ if and only if $\aalpha=\aalpha'$. % We state the main technical result of this section.

\begin{proposition}\label{prop:drift}
Consider \eqref{itm:bond}, assume that the environment is bounded \eqref{eq:bd} and let $0<\beta_-<\beta_+<\infty$, $\aalpha_+\in\M(\U)$ and 
\begin{align}
\Lambda(\beta_-,\beta_+,\aalpha_+)\coloneqq \left\{(\aalpha,\beta)\in\M(\U)\times[\beta_-,\beta_+]\colon d(\aalpha,\aalpha_+)\geq \rho_0^{\beta_-,\beta_+}(\beta,\beta_+)\right\}.
\end{align}
\begin{enumerate}
 \item[(i)] For every $f\colon \R_+\to\R$ convex,
\begin{align}
\sup_n\sup_{(\aalpha,\beta)\in\Lambda}\E\left[f(W_n^{\beta,\aalpha})\right]\leq \sup_n\E\left[f(W_n^{\beta_+,\aalpha_+})\right].
\end{align}
 \item[(ii)] If \eqref{eq:WD} holds for $(\aalpha_+,\beta_+)$ then there exists $p>1$ such that 
\begin{align}\label{eq:W_nW_infty}
\lim_{n\to\infty}\sup_{(\aalpha,\beta)\in\Lambda}\|W_n^{\beta,\aalpha}-W_\infty^{\beta,\aalpha}\|_p=0.
\end{align}
 Moreover, 
 \begin{align}\label{eq:W_nW_m}
\lim_{N\to\infty}\sup_{(\aalpha,\beta)\in\Lambda}\E\left[ \sup_{n,m\geq N}\left|W_n^{\beta,\aalpha}-W_m^{\beta,\aalpha}\right|\right]=0.
\end{align}
\end{enumerate}
\end{proposition}

Next, we have the following perturbation result for the free energy.

\begin{proposition}\label{prop:drift_free}
Under the assumptions of Proposition~\ref{prop:drift}, for every $\beta_1,\beta_2\in[\beta_-,\beta_+]$ with $\beta_1\leq\beta_2$ and every $\aalpha_1,\aalpha_2\in\M(\U)$ such that
\begin{align}\label{eq:thisthis}
d(\aalpha_1,\aalpha_2)\geq\rho_0^{\beta_-,\beta_+}(\beta_1,\beta_2),
\end{align}
it holds that 
\begin{align}
\p(\beta_1,\aalpha_1)\geq m(\aalpha_1|\aalpha_2)\p(\beta_2,\aalpha_2).
\end{align}
\end{proposition}

\subsection{Outline}
We first give the proofs for the auxiliary results: Theorem~\ref{thm:semigroup} can be found in Section~\ref{sec:proof_semi} and Propositions~\ref{prop:drift} and \ref{prop:drift_free} are proved in Section~\ref{sec:proof_drift}. The proofs for Theorems~\ref{thm:clt_prob}, \ref{thm:ldp} and \ref{thm:clt_as_pois} can be found in Sections~\ref{sec:proof_clt_prob}, \ref{sec:proof_ldp} and \ref{sec:proof_clt_as}.

\section{Proof of Theorem~\ref{thm:semigroup}}\label{sec:proof_semi}

We write 
\begin{align}\label{eq:h}
\h^\beta\coloneqq \exp\left(\beta \omega_0-\lambda(\beta)\right)
\end{align}
for the contribution of a generic instance of the environment and $h^\beta_{t,x}$, resp. $h^\beta_e$ for the contribution from site $(t,x)\in\N\times\Z^d$ resp. from bond $e\in E$. With this notation,
\begin{align}\label{eq:expression}
W_n^{\beta,\aalpha}(\omega)=
\begin{cases}
\sum_{\pi\in\S_n}P^\aalpha(\pi)\prod_{t=1}^n h^\beta_{t,\pi_t}&\text{ for }\eqref{itm:site},\\
\sum_{\pi\in\S_n}P^\aalpha(\pi)\prod_{e\in\pi} h^\beta_e&\text{ for }\eqref{itm:bond},
\end{cases}
\end{align}
where $\S_n$ denote the set of nearest-neighbor paths of length $n$ and where we identify at path with the set of its edges. We recall the relation $X\preceq_{cx}Y$, $X$ is smaller than $Y$ in the \emph{convex order}, between two real-valued random variables $X$ and $Y$, which holds if 
\begin{align*}
\E[f(X)]\leq\E[f(Y)]
\end{align*}
for all $f\colon\R\to\R$ convex such that both sides of the above inequality are well-defined. We refer to \cite{SS07} or to \cite{MS02} for surveys about $\preceq_{cx}$. In the next lemma, we show that Theorem~\ref{thm:semigroup} holds element-wise, from which the result follows by using the fact that $\preceq_{cx}$ is stable under taking mixtures.

\begin{lemma}\label{lem:aux}
Under the assumptions of Theorem~\ref{thm:semigroup}, there exists $C>0$ such that for every $\beta_1,\beta_2\in[\beta_-,\beta_+]$ with $\beta_1\leq\beta_2$,
\begin{align}\label{eq:to_do}
h^{\beta_1}\preceq_{cx}\rho_0 h^{\beta_2}+(1-\rho_0),
\end{align}
where $\rho_0(\beta_1,\beta_2)\coloneqq 1-C(\beta_2/\beta_1-1)$.
\end{lemma}

\begin{proof}[Proof of Theorem~\ref{thm:semigroup} for \eqref{itm:site} assuming Lemma~\ref{lem:aux}]
From \cite[Exercise 1.2]{peacocks} we obtain that, for any $\rho\in[\rho_0,1]$ and $\beta\geq 0$,
\begin{align*}
\rho_0 h^{\beta}+(1-\rho_0)\preceq_{cx}\rho h^{\beta}+(1-\rho).
\end{align*}
Thus, by Lemma~\ref{lem:aux} and the transitivity of $\preceq_{cx}$, for all $\beta_1,\beta_2\in[\beta_-,\beta_+]$ with $\beta_1\leq\beta_2$ and $\rho\geq \rho_0(\beta_1,\beta_2)$,
\begin{align*}
h^{\beta_1}\preceq_{cx}\rho h^{\beta_2}+(1-\rho).
\end{align*}
For any $(t,x)\in\N\times\Z^d$, by \cite[Theorem 3.A.4]{SS07}, there exists a coupling $(\widehat h_{t,x}^{\beta_1},\widehat h_{t,x}^{\beta_2})$ with the same marginals as $h^{\beta_1}$ and $h^{\beta_2}$ and such that
\begin{align*}
\widehat h_{t,x}^{\beta_1}=\rho\,\E\left[\widehat h_{t,x}^{\beta_2}\big|\widehat h_{t,x}^{\beta_1}\right]+(1-\rho).
\end{align*}
Since $\omega_0\mapsto h^\beta(\omega_0)$ is one-to-one, we obtain $(\omega_{t,x}^1,\omega_{t,x}^2)$ with $h^{\beta_i}_{t,x}(\omega^i_{t,x})=\widehat h^{\beta_i}_{t,x}$ for $i=1,2$, and it is clear that $\omega^i\coloneqq (\omega^i_{t,x})_{t\in\N,x\in\Z^d}$ has law $\Ps$ for $i=1,2$. For any $n\in\N$,
\begin{alignat*}{2}
\E\left[T_\rho W^{\beta_2,\aalpha}_n(\omega^2)\big|\omega^1\right]&=\sum_{\pi\in\S_n}P^\aalpha(\pi)\prod_{t=1}^n\E\left[\left(\rho h^{\beta_2}_{t,\pi(t)}(\omega^2)+(1-\rho)\right)\Big|\omega^1\right]&\\
&=\sum_{\pi\in\S_n}P^\aalpha(\pi)\prod_{t=1}^n\E\left[\big(\rho \widehat h^{\beta_2}_{t,\pi(t)}+(1-\rho)\big)\Big|\widehat h^{\beta_1}_{t,\pi(t)}\right]&\\
&=\sum_{\pi\in\S_n}P^\aalpha(\pi)\prod_{t=1}^n\widehat h^{\beta_1}_{t,\pi(t)}&\\&=W^{\beta_1,\aalpha}_n(\omega^1).
\end{alignat*}
This proves part (ii) and (iii), and part (i) follows from part (iii) by Jensen's inequality for conditional expectation.
\end{proof}

\begin{proof}[Proof of Lemma~\ref{lem:aux}]
To keep the notation is simple, we write $h(\beta)$ instead of $h^{\beta}$. Using \cite[Theorem 3.A.5]{SS07}, \eqref{eq:to_do} is equivalent to 
\begin{align}\label{eq:new_to_do}
\int_0^xF_{h(\beta_1)}^{-1}(u)\dd u\geq \int_0^xF_{\rho h(\beta_2)+(1-\rho)}^{-1}(u)\dd u\quad\text{ for all }x\in[0,1],
\end{align}
where $F_Z^{-1}$ denotes the right-continuous inverse of the distribution of a random variable $Z$. Note that
\begin{align*}
F_{h(\beta_1)}^{-1}(u)&=\frac{F_{e^{\beta_1\omega}}^{-1}(u)}{\E[e^{\beta_1\omega}]},\\
F_{\rho h(\beta_2)+(1-\rho)}^{-1}(u)&=\rho\frac{F_{e^{\beta_2\omega}}^{-1}(u)}{\E[e^{\beta_2\omega}]}+(1-\rho).
\end{align*}
We therefore have to check that, for all $x\in[0,1]$,
\begin{align}\label{eq:new_new_to_do}
\E[e^{\beta_2\omega}]\left(x\E[e^{\beta_1\omega}]-\int_0^xF_{e^{\beta_1\omega}}^{-1}(u)\dd u\right)\leq \rho\E[e^{\beta_1\omega}]\left(x\E[e^{\beta_2\omega}]-\int_0^xF_{e^{\beta_2\omega}}^{-1}(u)\dd u\right).
\end{align}
Observe that the functions $x\mapsto \int_0^xF_{e^{\beta_1\omega}}^{-1}(u)\dd u$ and $x\mapsto\int_0^xF_{e^{\beta_2\omega}}^{-1}(u)\dd u$ are convex. Since $\omega$ is not almost surely constant and $\beta_1>0$, they intersect the linear functions $x\mapsto x\E[e^{\beta_1\omega}]$ and $x\mapsto x\E[e^{\beta_2\omega}]$ exactly for $x=0$ and $x=1$. Hence both sides of \eqref{eq:new_new_to_do} are strictly positive for $x\in(0,1)$. We define
\begin{align}
\rho_0(\beta_1,\beta_2)&\coloneqq \sup_{x\in(0,1)}\frac{x\E[e^{\beta_1\omega}]\E[e^{\beta_2\omega}]-\E[e^{\beta_2\omega}]\int_0^xF_{e^{\beta_1\omega}}^{-1}(u)\dd u}{x\E[e^{\beta_1\omega}]\E[e^{\beta_2\omega}]-\E[e^{\beta_1\omega}]\int_0^xF_{e^{\beta_2\omega}}^{-1}(u)\dd u}\notag\\
&=1-\inf_{x\in(0,1)}\underbrace{\frac{\E[e^{\beta_2\omega}]\int_0^xF_{e^{\beta_1\omega}}^{-1}(u)\dd u-\E[e^{\beta_1\omega}]\int_0^xF_{e^{\beta_2\omega}}^{-1}(u)\dd u}{x\E[e^{\beta_1\omega}]\E[e^{\beta_2\omega}]-\E[e^{\beta_1\omega}]\int_0^xF_{e^{\beta_2\omega}}^{-1}(u)\dd u}}_{=:\frac{f(\beta_1,\beta_2,x)}{g(\beta_1,\beta_2,x)}}\label{eq:rho0}.
\end{align}
For \eqref{eq:new_to_do}, we have to show that there exists $C>0$ such that $\rho_0\leq 1-C(\beta_2/\beta_1-1)$. This will follow if we find $c_1,c_2>0$ such that, for all $x\in(0,1)$ and all $\beta_-<\beta_1<\beta_2<\beta_+$,
\begin{align}
f(\beta_1,\beta_2,x)&\geq c_1\left(\beta_2/\beta_1-1\right)(x\wedge (1-x)),\label{eq:f}\\
g(\beta_1,\beta_2,x)&\leq c_2(x\wedge (1-x)).\label{eq:g}
\end{align}
Using $F_{e^{\beta_2\omega}}^{-1}(u)=(F_{e^{\beta_1\omega}}^{-1}(u))^{\beta_2/\beta_1}$, we get
\begin{align*}
f(\beta_1,\beta_2,x)&=\E[e^{\beta_2\omega}]\int_0^xF_{e^{\beta_1\omega}}^{-1}(u)\dd u-\E[e^{\beta_1\omega}]\int_0^xF_{e^{\beta_2\omega}}^{-1}(u)\dd u\\
&=\int_0^xF_{e^{\beta_1\omega}}^{-1}(u)\dd u\int_0^1F_{e^{\beta_2\omega}}^{-1}(v)\dd v-\int_0^1F_{e^{\beta_1\omega}}^{-1}(u)\dd u\int_0^xF_{e^{\beta_2\omega}}^{-1}(v)\dd v\\
&=\int_0^xF_{e^{\beta_1\omega}}^{-1}(u)\dd u\int_0^1(F_{e^{\beta_1\omega}}^{-1}(v))^{\beta_2/\beta_1}\dd v-\int_0^1F_{e^{\beta_1\omega}}^{-1}(u)\dd u\int_0^x(F_{e^{\beta_1\omega}}^{-1}(v))^{\beta_2/\beta_1}\dd v\\
&=\int_{[0,1]^d}\ind_{\{u<x<v\}}F_{e^{\beta_1\omega}}^{-1}(u)F_{e^{\beta_1\omega}}^{-1}(v)\left((F_{e^{\beta_1\omega}}^{-1}(v))^{\beta_2/\beta_1-1}-(F_{e^{\beta_1\omega}}^{-1}(u))^{\beta_2/\beta_1-1}\right)\dd u\dd v.
\end{align*}
Note that the integrand is non-negative since $\beta_2\geq\beta_1$ and $F_{e^{\beta_1\omega}}^{-1}$ is increasing. Choose $\eps>0$ small enough that
\begin{align*}
a\coloneqq \sup_{\beta\in[\beta_-,\beta_+]}F_{e^{\beta\omega}}^{-1}(\eps)<\inf_{\beta\in[\beta_-,\beta_+]}F_{e^{\beta\omega}}^{-1}(1-\eps)=:b.
\end{align*}
Then for all $x\leq\eps$,
\begin{align*}
f(\beta_1,\beta_2,x)&\geq \left(b^{\beta_2/\beta_1-1}-a^{\beta_2/\beta_1-1}\right)\int_0^xF_{e^{\beta_1\omega}}^{-1}(u)\dd u\int_{1-\eps}^1 F_{e^{\beta_1\omega}}^{-1}(u)\dd u\\
&\geq x\left(b^{\beta_2/\beta_1-1}-a^{\beta_2/\beta_1-1}\right) (1-\eps)e^{-2K\beta_+}\\
&\geq x\left((b/a)^{\beta_2/\beta_1-1}-1\right) e^{-(K+1)\beta_+/\beta_-}(1-\eps)e^{-2K\beta_+}\\
&\geq x\left(\beta_2/\beta_1-1\right)\log (b/a)e^{-(K+1)\beta_+/\beta_-}(1-\eps)e^{-2K\beta_+}.
\end{align*}
Similar calculations show
\begin{align*}
f(\beta_1,\beta_2,x)&\geq (1-x)(\beta_2/\beta_1-1)c\quad\text{ for }x\geq 1-\eps,\\
f(\beta_1,\beta_2,x)&\geq (\beta_2/\beta_1-1)c\quad\text{ for }x\in[\eps,1-\eps].
\end{align*}
Combining these claims gives \eqref{eq:f}. It remains to check \eqref{eq:g}, for which we observe that $g$ is concave and differentiable with
\begin{align*}
g'(0)&=\E[e^{\beta_1\omega}]\left(\E[e^{\beta_2(\omega-\essinf\omega)}]-1\right)\essinf e^{\beta_2\omega},\\
g'(1)&=-\E[e^{\beta_1\omega}]\left(1-\E[e^{\beta_2(\omega-\esssup\omega)}]\right)\esssup e^{\beta_2\omega}.
\end{align*}
Since $\omega$ is bounded and not almost surely constant, it is not hard to see that 
\begin{align*}
\sup_{\beta_-\leq\beta_1\leq\beta_2\leq\beta_+}g'(0)&<\infty,\\
\inf_{\beta_-\leq\beta_1\leq\beta_2\leq\beta_+}g'(0)&>-\infty,
\end{align*}
and this concludes the proof.\end{proof}

\begin{remark}\label{rem:lorenz}
We quickly illustrate why we have to assume boundedness \eqref{eq:bd} in Theorem~\ref{thm:semigroup}. It is related to the so-called \emph{Lorenz curve} of a random variable $X$, which is the function defined by 
\begin{align*}
L_X(x)=\frac{\int_0^xF_{X}^{-1}(u)\dd u}{\E[X]},
\end{align*}
see the discussion in \cite[Chapter 17.C]{majorization}. Two examples are depicted in Figure~\ref{fig:lorenz}. If $X$ is not almost surely constant then $L_X$ is a concave function intersecting the diagonal exactly at $x=0$ and $x=1$. The extent to which $L_X$ lies below the diagonal is related to the ``variability'' of $X$, which in the original context is interpreted as a measure of inequality in the distribution of  wealth within a society. Notice that $\rho_0$ from \eqref{eq:rho0} can be written as
\begin{align*}
1-\inf_{x\in(0,1)}\frac{L_{e^{\beta_1\omega}}(x)-L_{e^{\beta_2\omega}}(x)}{x-L_{e^{\beta_2\omega}}(x)}.
\end{align*}
As is easy to guess from Figure~\ref{fig:lorenz}, the infimum is positive exactly if the derivative of $L_{e^{\beta_1\omega}}$ at $0$ and $1$ is bounded away from $0$ and $\infty$, which is equivalent to \eqref{eq:bd}.
\end{remark}

\begin{figure}[h!]
\includegraphics[width=.45\textwidth]{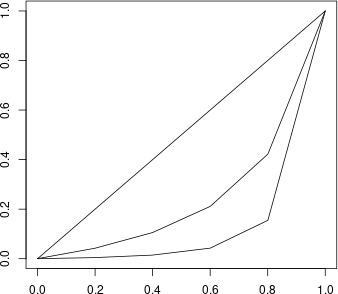}
\includegraphics[width=.45\textwidth]{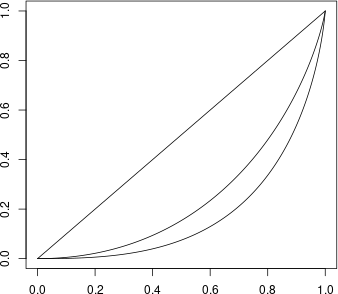}
\caption{The middle lines are the Lorenz curves (see Remark~\ref{rem:lorenz}) of a random variable $X$ whose support is bounded away from $0$ and $\infty$ (left picture) and of an exponential random variable (right picture). The bottom lines are the Lorenz curves of $X^2$.}\label{fig:lorenz}
\end{figure}

\section{Proof of the perturbative results for bond disorder}\label{sec:proof_drift}

Throughout this section we consider \eqref{itm:bond}. Let $E_u$ denote the set of edges in direction $u\in\U$,
\begin{align}\label{eq:Eu}
E_u\coloneqq \left\{\big((i,x),(i+1,x+u)\big):i\in\N,x\in\Z^d\right\}.
\end{align}

\subsection{A comparison lemma}

The main technical ingredient in the proof of Proposition~\ref{prop:drift} is Lemma~\ref{lem:strange_coupling} below, where we construct a coupling to relate $T_\rrho W_n^{\beta,\aalpha}$ to $W_n^{\beta,\aalpha'}$, for some inhomogeneous resampling parameter $\rrho$ depending on $\beta,\aalpha,\aalpha'$. The idea is that $\rrho_e$ depends on the direction of edge $e$, i.e. on  $u\in\U$ such that $e\in E_u$. We need to introduce some notation to identify $T_\rrho W_n^{\beta,\aalpha}$ with a mixture of $W_n^{\beta,\aalpha'}$.

\smallskip Recall that $\S_n$ denotes the set of nearest-neighbor paths $\pi$ of length $n$, i.e. $|\pi_{i}-\pi_{i-1}|_1=1$ for $i=1,\dots,n$, with $\S\coloneqq \S_\infty$. We write $\S_n^0$, resp. $\S^0$, for the set of paths $\pi$ that are allowed to stay in place, i.e. $|\pi_{i+1}-\pi_i|_1\leq 1$ for all $i$, and $P^\eeta$ for the law of a random walk with increment distribution $\eeta\in\M(\U\cup\{0\})$, defined similarly to \eqref{eq:RW_drift}. For $\pi\in\S^0$, we define the set of non-jump times of $\pi$,
\begin{align*}
N_n(\pi)&\coloneqq \{i\in\n:\pi_i=\pi_{i+1}\},
\end{align*}
where $\n=\{0,dots ,n\}$. For an edge $e=\big((i,x),(i+1,x+u)\big)\in E$, we introduce $e(\pi)\coloneqq \left((i',x'),(i'+1,x'+u)\right)\in E$ by
\begin{align*}
i'=\inf\{n:|N_\infty(\pi)\cap\n|> i\}\quad\text{ and }\quad x'=x+\pi_{i'}.
\end{align*}
Finally, given $\omega\in\Omegab$ and $\pi\in\S^0$, we define a new environment $\omega(\pi)\in\Omegab$ by $(\omega(\pi))_e\coloneqq \omega_{e(\pi)}$. In words, $\omega(\pi)$ is defined by deleting the time-slices where $\pi$ jumps and by spatially shifting the remaining time-slices according to the jumps of $\pi$. See Figure~\ref{fig:coupling} for an illustration. We stress that $\omega(\pi)$ has law $\Pb$ for every fixed $\pi\in\S_n^0$.

\begin{figure}[b]
\includegraphics[width=\textwidth]{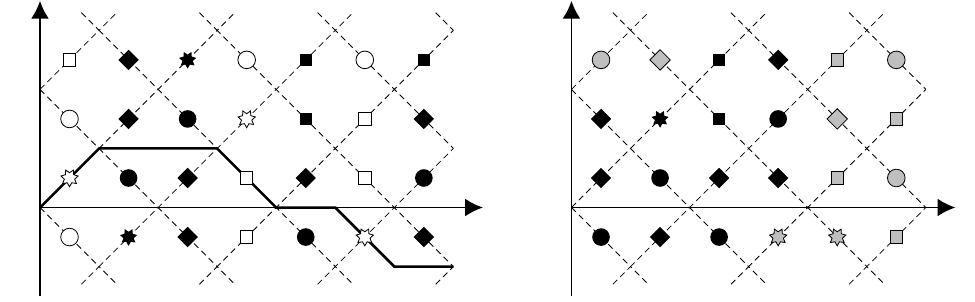}
\caption{Illustration of the coupling between $\omega$ (left), $\pi$ (thick path on the left) and $\omega(\pi)$ (right). Different values of the environment are represented by symbols on the bonds (dashed lines), with colors indicating whether the time-slice will be deleted (white) or not (black). The gray symbols on the right are filled up from outside the visible area on the left.}\label{fig:coupling}
\end{figure}

\begin{lemma}\label{lem:strange_coupling}
Consider \eqref{itm:bond} and recall \eqref{eq:def_m}. Let $\ggamma,\ggamma'\in\M(\U)$ and define $\rrho(\ggamma\to\ggamma')=(\rrho_e)_{e\in E}\in(0,1]^E$ and $\eeta(\ggamma\to\ggamma')\in\M(\U\cup\{0\})$ by
\begin{align*}
\rrho_e&\coloneqq \frac{\ggamma_u'}{\ggamma_u}m(\ggamma|\ggamma')\quad\text{ for }e\in E_u,\\
\eeta_u&\coloneqq \ggamma_u-\ggamma_u'm(\ggamma|\ggamma')\quad\text{ for }u\in\U,\\
\eeta_0&\coloneqq m(\ggamma|\ggamma').
\end{align*}
Then, for any $\omega\in\Omegab$, $n\in\N$ and $\beta\geq 0$,
\begin{align}\label{eq:strange}
\sum_\pi P^{\eeta}(\pi)W_{|N_n(\pi)|}^{\beta,\ggamma'}(\omega(\pi))= \left(T_\rrho W_{n}^{\beta,\ggamma}\right)(\omega).
\end{align}
\end{lemma}

\begin{proof}
We have to introduce further notation that will not be used outside of this proof. For $\pi\in\S_n^0$ and $\sigma\in\S_n$, we say that $\sigma$ obeys $\pi$, $\sigma\ll_n\pi$, if $\pi$ and $\sigma$ have the same non-zero jumps,
\begin{align*}
\sigma\ll_n\pi\quad\iff\quad\pi_{i+1}-\pi_i=\sigma_{i+1}-\sigma_i\text{ for all }i\in\n\setminus N_n(\pi).
\end{align*}
Recall that we identify a path with the set of its bonds. For $\sigma\ll_n\pi$, we write $\sigma\setminus\pi$ for the set of bonds of $\sigma$ during time steps where $\pi$ did not jump, i.e.
\begin{align*}
\big((i,x),(i+1,x+u)\big)\in\sigma\setminus\pi\quad\iff\quad \sigma_i=x,\sigma_{i+1}=x+u\text{ and }i\in N_n(\pi).
\end{align*}
Note that $|\sigma\setminus\pi|=|N_n(\pi)|$. The proof is now a simple calculation. 
\begin{align*}
\left(T_\rrho W_{n}^{\beta,\ggamma}\right)(\omega)&=\sum_{\sigma\in\S_n} \prod_{u\in\U}(\ggamma_u)^{|E_u\cap\sigma|}\prod_{e\in E_u\cap\sigma} \left(\rrho_eh^\beta_e(\omega)+(1-\rrho_e)\right)\\
&=\sum_{\pi\in\S_n^0}\Big(\prod_{u\in\U}\big(\ggamma_u(1-\rrho_u)\big)^{|E_u\cap \pi|}\Big)\sum_{\sigma\in\S_n\colon\sigma\ll_n\pi}\prod_{u\in\U}(\ggamma_u\rrho_u)^{|E_u\cap (\sigma\setminus\pi)|}\prod_{e\in E_u\cap(\sigma\setminus\pi)} h^\beta_e(\omega)\\
&=\sum_{\pi\in\S_n^0}\underbrace{\Big(\prod_{u\in\U}\eeta_u^{|E_u\cap \pi|}\Big)\eeta_0^{|N_n(\pi)|}}_{=P^{\eeta}(\pi)}\underbrace{\sum_{\sigma\in\S_n\colon\sigma\ll_n\pi}\prod_{u\in\U}(\ggamma_u')^{|E_u\cap (\sigma\setminus\pi)|}\prod_{e\in E_u\cap(\sigma\setminus\pi)} h^\beta_e(\omega)}_{=W_{|N_n(\pi)|}^{\beta,\ggamma'}(\omega(\pi))}.\tag*{\qedhere}
\end{align*}
\end{proof}

\subsection{Proof of Propositions~\ref{prop:drift} and~\ref{prop:drift_free}}

\begin{proof}[Proof of Proposition~\ref{prop:drift}]
Fix $(\aalpha,\beta)\in\Lambda$ and let $\rrho=\rrho(\aalpha\to\aalpha_+)$ and $\eeta=\eeta(\aalpha\to\aalpha_+)$ be as defined in Lemma~\ref{lem:strange_coupling}. By definition, we have, for every $e\in E$,
\begin{align*}
\rrho_e\geq d(\aalpha,\aalpha_+)\geq \rho_0^{\beta_-,\beta_+}(\beta,\beta_+),
\end{align*}
so Theorem~\ref{thm:semigroup} implies
\begin{align*}
\E\left[f(W_n^{\beta,\aalpha})\right]\leq \E\left[f(T_{\rrho}W_n^{\beta_+,\aalpha})\right].
\end{align*}
Moreover, by Lemma~\ref{lem:strange_coupling} and Jensen's inequality,
\begin{align*}
\E\left[f(T_{\rrho}W_n^{\beta_+,\aalpha})\right]&=\E\left[f\Big(E^{\eeta}\left[W_{|N_n(\Pi)|}^{\beta_+,\aalpha_+}(\omega(\Pi))\right]\Big)\right]\\
&\leq \E\otimes E^{\eeta}\left[f\big(W_{|N_n(\Pi)|}^{\beta_+,\aalpha_+}(\omega(\Pi))\big)\right]\\
&=\sum_{k=0}^n P\big(\text{Bin}(n,q)=k)\E\left[f\big(W_{k}^{\beta_+,\aalpha_+}(\omega)\big)\right]\\
&\leq \sup_k\E\left[f\big(W_{k}^{\beta_+,\aalpha_+}(\omega)\big)\right], 
\end{align*}
where $\Pi$ is the random path chosen according to $P^{\eeta}$ and $q=m(\aalpha|\aalpha_+)$. This finishes the proof of part~(i). For \eqref{eq:W_nW_infty}, we first note that, by \eqref{eq:Lp_bd}  and the $L^p$-martingale convergence theorem, there exists $p>1$ such that
\begin{align}\label{eq:Lp}
\lim_{n\to\infty}\|W_n^{\beta_+,\aalpha_+}-W_\infty^{\beta_+,\aalpha_+}\|_p=0.
\end{align}
Let $(\omega_1,\omega_2)$ denote the coupling from \eqref{eq:semigroup_coupling}, so that, for any $n\in\N$,
\begin{align*}
W_n^{\beta,\aalpha}(\omega_1)&=\E\left[\big(T_\rrho W_n^{\beta_+,\aalpha}(\omega_2)\big)\Big|\omega_1\right]\\
&=\E\otimes E^{\eeta}\left[W^{\beta_+,\aalpha_+}_{|N_n(\Pi)|}(\omega_2(\Pi))\big)\Big|\omega_1\right].
\end{align*}
The second equality is Lemma~\ref{lem:strange_coupling}. We have 
\begin{align*}
W^{\beta_+,\aalpha_+}_{|N_n(\Pi)|}(\omega_2(\Pi))\leq \sup_kW^{\beta_+,\aalpha_+}_{k}(\omega_2(\Pi))
\end{align*}
and the right hand side is $\P\otimes P^{\eeta}$-integrable by \eqref{eq:sup_integrable}. Lebesgue's convergence theorem for conditional expectations, together with the fact that $P^{\eeta}$-almost surely $\lim_{n\to\infty}|N_n(\Pi)|=\infty$ , imply
\begin{align*}
W_\infty^{\beta,\aalpha}(\omega_1)=\E\otimes E^{\eeta}\left[W^{\beta_+,\aalpha_+}_{\infty}(\omega_2(\Pi))\Big|\omega_1\right].
\end{align*}
Using \eqref{eq:Lp} and Jensen's inequality, we can now conclude:
\begin{align*}
\E\left[\big|W_n^{\beta,\aalpha}(\omega_1)-W_\infty^{\beta,\aalpha}(\omega_1)\big|^p\right]&=\E\left[\Big|\E\otimes E^{\eeta}\Big[W_{|N_n(\omega)|}^{\beta_+,\aalpha_+}(\omega_2(\Pi))-W_\infty^{\beta_+,\aalpha_+}(\omega_2(\Pi))\Big|\omega_1\Big]\Big|^p\right]\\
&\leq \sum_{k=0}^n P\big(\text{Bin}(n,q)=k\big)\E\left[\big|W_{k}^{\beta_+,\aalpha_+}-W_\infty^{\beta_+,\aalpha_+}\big|^p\right]\\
&\leq \sup_{k\geq nq/2}\|W_{k}^{\beta_+,\aalpha_+}-W_\infty^{\beta_+,\aalpha_+}\|_p^p\\
&\qquad+2P\left(\text{Bin}(n,q)\leq nq/2\right)\sup_k\|W_k^{\beta_+,\aalpha_+}\|_p^p\\
&\xrightarrow{n\to\infty}0.
\end{align*}
The argument for \eqref{eq:W_nW_m} is similar. We have 
\begin{align*}
\E\left[ \sup_{n,m\geq N}\left|W_n^{\beta,\aalpha}-W_m^{\beta,\aalpha}\right|\right]&=\E\left[ \sup_{n,m\geq N}\Big|\E\left[(T_\rrho W_n^{\beta_+,\aalpha})(\omega_2)-(T_\rrho W_m^{\beta_+,\aalpha})(\omega_2)\big|\omega_1\right]\Big|\right]\\
&\leq \E\left[ \sup_{n,m\geq N}\left|T_\rrho W_n^{\beta_+,\aalpha}-T_\rrho W_m^{\beta_+,\aalpha}\right|\right]\\
&=\E\left[ \sup_{n,m\geq N}\left|E^{\eeta}\hspace{-1mm}\left[W_{N_n(\Pi)}^{\beta_+,\aalpha_+}(\omega(\Pi))-W_{N_m(\pi)}^{\beta_+,\aalpha_+}(\omega(\Pi))\right]\right|\right]\\
&\leq \E\otimes E^{\eeta}\hspace{-1mm}\left[\sup_{n,m\geq N}\left|W_{N_n(\Pi)}^{\beta_+,\aalpha_+}(\omega(\Pi))-W_{N_m(\pi)}^{\beta_+,\aalpha_+}(\omega(\Pi))\right|\right]\\
&\leq \E\left[\sup_{n,m\geq Nq/2}\left|W_{n}^{\beta_+,\aalpha_+}-W_{m}^{\beta_+,\aalpha_+}\right|\right]\\
&\qquad +2\P\left(\operatorname{Bin}(N,q)\leq Nq/2\right)\E\left[\sup_n W_n^{\beta_+,\aalpha_+}\right].
\end{align*}
In the first and second equalities we used \eqref{eq:semigroup_coupling} and \eqref{eq:strange} while the first and second inequalities follow from Jensen's inequality and the monotone convergence theorem. In the final inequality, we used 
$q\coloneqq\inf_{(\aalpha,\beta)\in\Lambda}m(\aalpha|\aalpha_+)>0$. Hence, by the dominated convergence theorem and \eqref{eq:sup_integrable},
\begin{align*}
\lim_{N\to\infty}\sup_{(\aalpha,\beta)\in\Lambda}\E\left[ \sup_{n,m\geq N}\left|W_n^{\beta,\aalpha}-W_m^{\beta,\aalpha}\right|\right]\leq \lim_{N\to\infty}\E\left[\sup_{n,m\geq N}\left|W_n^{\beta_+,\aalpha_+}-W_m^{\beta_+,\aalpha_+}\right|\right]=0.\tag*{\qedhere}
\end{align*}
\end{proof}

\begin{proof}[Proof of Proposition~\ref{prop:drift_free}]
Let $\rrho=\rrho(\aalpha_1\to\aalpha_2)$ and $\eeta=\eeta(\aalpha_1\to\aalpha_2)$ be as defined in Lemma~\ref{lem:strange_coupling}. By assumption, $\rrho\geq d(\aalpha_1,\aalpha_2)\geq \rho_0^{\beta_-,\beta_+}(\beta_1,\beta_2)$ and therefore, using Theorem~\ref{thm:semigroup} and Lemma~\ref{lem:strange_coupling},
\begin{align*}
\E\left[\log W_n^{\beta_1,\aalpha_1}(\omega)\right]&\geq \E\left[\log T_\rrho W_n^{\beta_2,\aalpha_1}(\omega)\right]\\
&=\E\left[\log E^\eeta\big[W^{\beta_2,\aalpha_2}_{|N_n(\Pi)|}(\omega(\Pi))\big]\right]\\
&\geq \sum_{k=0}^nP\left(\text{Bin}(n,q)=k\right)\E\left[\log W^{\beta_2,\aalpha_2}_{k}(\omega)\right],
\end{align*}
where $q=m(\aalpha_1|\aalpha_2)$. Standard concentration properties of the binomial distribution around its mean then show $\p(\beta_1,\aalpha_1)\geq q\,\p(\beta_2,\aalpha_2)$.
\end{proof}

\section{Proof of the main results}\label{sec:proof_main}

\subsection{Proof of the CLT in probability}\label{sec:proof_clt_prob}

\begin{proof}[Proof of Theorem~\ref{thm:clt_prob}(i)]
We apply Proposition~\ref{prop:drift}(i) with $\beta_-\coloneqq\beta$, $\beta_+\coloneqq \frac{\beta+\beta_{cr}}2$ and $\alpha_+\coloneqq \1/2d$. We have $\rho_0^{\beta_-,\beta_+}(\beta,\beta_+)<1$ by Theorem~\ref{thm:semigroup}(ii) and $\aalpha\mapsto d(\aalpha,\aalpha_+)$ is continuous, so we find $\alpha_0>0$ such that
\begin{align*}
\left(\M(\U)\cap(\1/2d+[-\alpha_0,\alpha_0]^d)\right)\times\{\beta\}\subseteq\Lambda\tag*{\qedhere}
\end{align*}
\end{proof}

For Theorem~\ref{thm:clt_prob}(ii) we need the following lemma.

\begin{lemma}\label{lem:mgf_in_prob}
Recall \eqref{eq:alpha_lambda}. Under the assumptions of Theorem~\ref{thm:clt_prob}, for any $\llambda\in\R^d$,
\begin{align}\label{eq:inprob}
W_n^{\beta,\aalpha(\llambda/\sqrt n)}\xrightarrow[n\to\infty]PW_\infty^{\beta}.
\end{align}
Moreover, for any $K>1$,
\begin{align}\label{eq:tightinprob}
\lim_{n\to\infty}\P\left(\mu_{\omega,n}^\beta(X_n\cdot e_1\geq K\sqrt n)>2e^{-K^2/2}\right)=0.
\end{align}
\end{lemma}

\begin{proof}[Proof of Theorem~\ref{thm:clt_prob}(ii)]
By a well-known criterion for convergence in probability in metric spaces, it is enough to show that any subsequence $(n_k)_{k\in\N}$ has a sub-subsequence $(n_{k_l})_{l\in\N}$ such that almost surely
\begin{align*}
\lim_{l\to\infty}d_P\left(\mu_{\omega,n_{k_l}}^\beta\Big(\frac{X_{n_{k_l}}}{\sqrt {n_{k_l}}}\in\cdot \Big),\mathcal N\right)=0.
\end{align*}
We show  that we can find $(n_{k_l})_{l\in\N}$ such that almost surely,
\begin{itemize}
 \item For every $K\in\N$ and for all but finitely many $l$,
 \begin{align*}
\mu_{\omega,n_{k_l}}^\beta\left(|X_{n_{k_l}}|>K\sqrt {n_{k_l}} \right)\leq 4de^{-K^2/2}.
\end{align*}
 \item For every $\llambda\in\Q^d$, 
 \begin{align*}
\lim_{l\to\infty}\mu_{\omega,n_{k_l}}^\beta\left[e^{\llambda\cdot X_{n_{k_l}}/\sqrt {n_{k_l}}}\right]=e^{\frac 12\|\llambda\|_2^2}.
\end{align*}
\end{itemize}
Indeed, the first condition shows that $\big\{\mu_{\omega,n_{k_l}}^\beta({X_{n_{k_l}}}/{\sqrt {n_{k_l}}}\in\cdot )\colon l\in\N\big\}$ is tight and the second condition that the moment generating function of any weak accumulation point agrees with that of the normal distribution on $\Q^d$, and hence on $\R^d$. 

\smallskip Let $\llambda_1,\llambda_2,\dots$ be an enumeration of $\Q^d$. Let $k_1\coloneqq 1$ and, given $k_{l-1}$, choose $k_{l}\geq k_{l-1}$ large enough that
\begin{align}
\P\left(\mu_{\omega,n_{k_{l}}}^\beta (|X_{n_{k_{l}}}|\geq K\sqrt {n_{k_{l}}})<4de^{-K^2/2}\text{ for all }K=1,\dots,l\right)&\geq 1-\frac 1{l^2},\label{eq:first}\\
\P\left(\left|\mu_{\omega,n_{k_{l}}}\left[e^{\llambda_i\cdot X_{n_{k_{l}}}/\sqrt{n_{k_{l}}}}\,\right]-E\left[e^{\llambda_i\cdot X_{n_{k_{l}}}/\sqrt{n_{k_{l}}}}\,\right]\right|<\frac{1}{l}\text{ for all }i=1,\dots,l\right)&\geq 1-\frac 1{l^2}.\label{eq:second}
\end{align}
The claim then follows by the Borel-Cantelli Lemma and the central limit theorem for the simple random walk. Indeed, \eqref{eq:first} follows from \eqref{eq:tightinprob} and  we observe that
\begin{align}\label{eq:drift_mgf}
\mu_{\omega,n_{k_{l}}}\left[e^{\llambda_i\cdot X_{n_{k_{l}}}/\sqrt{n_{k_{l}}}}\right]=\frac{W^{\beta,\aalpha(\llambda_i/\sqrt {n_{k_l}})}_{n_{k_l}}}{W^{\beta}_{n_{k_l}}}E\left[e^{\llambda_i\cdot X_{n_{k_{l}}}/\sqrt{n_{k_{l}}}}\,\right].
\end{align}
Thus \eqref{eq:second} follows from \eqref{eq:inprob}.
\end{proof}

\begin{proof}[Proof of Lemma~\ref{lem:mgf_in_prob}]
Let $\delta>0$. By \eqref{eq:W_nW_infty}, there exists $N\in\N$ such that
\begin{align*}
\sup_{n\geq N}\sup_{\aalpha\in(\1/2d+[-\alpha_0,\alpha_0]^d)\cap\M(\U)}\|W_n^{\beta,\aalpha}-W_\infty^{\beta,\aalpha}\|_1\leq \frac{\delta\eps}{3}.
\end{align*}
Note that $\aalpha(\llambda/\sqrt n)\in\1/2d+[-\alpha_0(\beta),\alpha_0(\beta)]^d$ for $n$ large enough. Thus, for any $n\geq N$ large enough,
\begin{align*}
\P\left(\left|W^{\beta,\aalpha(\llambda /\sqrt n)}_n-W^{\beta}_\infty\right|>\eps \right)
&\leq \P\left(\left|W^{\beta,\aalpha(\llambda /\sqrt n)}_{N}-W^{\beta}_{N}\right|>\eps \right)+\delta.
\end{align*}
The claim now follows because $\llambda\mapsto W_N^{\beta,\aalpha(\llambda)}(\omega)$ is continuous for fixed $N$. For the second claim, by the central limit theorem for the simple random walk and \eqref{eq:inprob},
\begin{align*}
\mu_{\omega,n}^\beta(X_n\cdot e_1\geq K\sqrt n)&=\mu_{\omega,n}^\beta(e^{KX_n\cdot e_1/\sqrt n}\geq e^{K^2})\\
&\leq e^{-K^2}E[e^{X\cdot e_1K/\sqrt n}]\frac{W^{\beta,\aalpha(Ke_1/\sqrt n)}_n}{W_n^{\beta}(\omega)}
\\
&\xrightarrow[n\to\infty]Pe^{-K^2/2}.\tag*{\qedhere}
\end{align*}
\end{proof}

\begin{remark}\label{rm:invariance}
We briefly discuss the steps needed to obtain a full invariance principle, i.e., that 
\begin{align*}
d_P\Big( \mu_{\omega,n}^\beta\big((X_{tn}/\sqrt n)_{t\in[0,1]}\in\cdot\big), \PBM\Big)\xrightarrow[n\to\infty]P0,
\end{align*}
where $d_P$ now denotes the Prokhorov metric on $C([0,1],\R^d)$, the space of continuous functions equipped with the uniform topology. First, to show that the marginals $(X_{t_1n}/\sqrt n,dots ,X_{t_kn}/\sqrt n)$ converge to the corresponding marginals of $P^{BM}$, we need a time-inhomogeneous version of \eqref{eq:inprob}, i.e., with the underlying random walk $P^{\aalpha(\llambda/\sqrt n)}$ replaced by a random walk with different (vanishing) drifts in each of the intervals $[0,t_1n],[t_1n,t_2n],dots ,[t_{k-1}n,t_kn]$. This requires a straightforward generalization of Lemma~\ref{lem:strange_coupling} to a time-dependent resampling rate $\rrho$. From there, the invariance principle follows by proving tightness, i.e., that every subsequence $(n_k)_{k\in\N}$ has a sub-subsequence $(n_{k_l})_{l\in\N}$ such that 
\begin{align*}
\mu_{\omega,n_{k_l}}^\beta\Big(\Big(\frac{X_{\lfloor t{n_{k_l}}\rfloor}}{\sqrt {n_{k_l}}}\Big)_{t\in[0,1]}\in\cdot \Big)
\end{align*}
is almost surely tight. This can be proved by following the argument in \cite[Section 4.2]{G21}.
\end{remark}

\subsection{Proof of the LDP}\label{sec:proof_ldp}

\begin{proof}[Proof of Theorem~\ref{thm:ldp}]
Since $\beta<\overline{\beta}_{cr}$, we find $\eps>0$ small enough that $\p(\beta+\eps)=0$. For part (i), we apply Proposition~\ref{prop:drift_free} with $\beta_1=\beta_-\coloneqq\beta$, $\beta_2=\beta_+\coloneqq\beta+\eps$ and $\aalpha_2=\1/2d$. By Theorem~\ref{thm:semigroup}(ii), we have $\rho^{\beta_-,\beta_+}_0(\beta_1,\beta_2)<1$. Since $\aalpha\mapsto d(\aalpha,\aalpha_2)$ is continuous, we find $\alpha_0>0$ such that \eqref{eq:thisthis} holds for all $\aalpha_1\in (\1/2d+[-\alpha_0,\alpha_0]^d)\cap\M(\U)$. Thus
\begin{align*}
0\geq \p(\beta_1,\aalpha_1)\geq (1-m(\aalpha_1|\aalpha_2))\p(\beta_2,\aalpha_2)=0.
\end{align*}
Recall \eqref{eq:alpha_lambda}. Part (ii) follows from the G\"artner-Ellis theorem \cite[Theorem 2.3.6]{DZ} and
\begin{align}
\lim_{n\to\infty}\frac 1n\log \mu_{\omega,n}^\beta[e^{\llambda\cdot X_n}]&=\lim_{n\to\infty}\frac 1n\log E[e^{\llambda\cdot X_n}]+\p(\beta,\aalpha(\llambda))-\p(\beta)\label{eq:this_eq}\\
&=\lim_{n\to\infty}\frac 1n\log E[e^{\llambda\cdot X_n}],\notag
\end{align}
where the second equality holds for all $\llambda\in\R^d$ small enough by part (i). Finally, for part (iii), we note $\p(\beta)=0$ for $\beta\leq\overline\beta_{cr}$ and $\p(\beta,\aalpha)\leq 0$ for all $\beta\geq 0$ and $\aalpha\in\M(\U)$. Hence \eqref{eq:this_eq} implies
\begin{align*}
\lim_{n\to\infty}\frac 1n\log \mu_{\omega,n}^\beta[e^{\llambda\cdot X_n}]\leq \lim_{n\to\infty}\frac 1n\log E[e^{\llambda\cdot X_n}].
\end{align*}
The conclusion again follows by the G\"artner-Ellis theorem.
\end{proof}

\subsection{Proof of the almost sure CLT}\label{sec:proof_clt_as}

We first prove a version of the ``defective'' result, Theorem~\ref{thm:clt_as_bond}, for the Poissonian case. To obtain Theorem~\ref{thm:clt_as_pois}, we then use the shift invariance of the Brownian polymer.

\begin{proposition}\label{prop:defect}
Under the assumptions of Theorem~\ref{thm:clt_as_bond}, for $\Pp$-almost all $\omega$ there exists a Borel set $\Lambda(\omega)\subseteq[0,1]^d$ such that $|\Lambda(\omega)|=1$ and such that, for all $\llambda\in \Lambda(\omega)$,
\begin{align*}
\lim_{t\to\infty}d_P\left(\mu_{\omega,t}^{\beta,\llambda}\Big(\frac{X_t-\llambda t}{\sqrt t}\in\cdot\Big),\mathcal N\right)=0.
\end{align*}
\end{proposition}

\begin{proof}[Proof of Theorem~\ref{thm:clt_as_pois} assuming Proposition~\ref{prop:defect}]
Recall the definition of $\omega(\llambda)$ from \eqref{eq:omega_alpha} and note that 
\begin{align*}
\mu_{\omega,t}^{\beta,\llambda}\Big(\frac{X_t-\llambda t}{\sqrt t}\in\cdot \Big)=\mu_{\omega(\llambda),t}^{\beta}\Big(\frac{X_t}{\sqrt t}\in\cdot \Big).
\end{align*}
Hence $\0\in \Lambda(\omega(\llambda))\iff \llambda\in \Lambda(\omega)$ and, since $\omega(\llambda)$ has the same law as $\omega$ for any fixed $\llambda\in\R^d$,
\begin{align*}
&\P\left(\lim_{t\to\infty}d_P\left(\mu_{\omega,t}^\beta\Big(\frac{X_t}{\sqrt t}\in\cdot\Big),\mathcal N\right)=0\right)\geq \P(\0\in \Lambda(\omega))=\int_{[0,1]^d}\P(\llambda\in \Lambda(\omega))\,\dd\llambda=1.
\tag*{\qedhere}\end{align*}
\end{proof}

We prove Proposition~\ref{prop:defect} using an almost sure version of Lemma~\ref{lem:mgf_in_prob}, so we need to exclude the possibility that $W_t^{\beta,\llambda+\llambda'/\sqrt t}$ does not converge to $W_\infty^{\beta,\llambda}$ for a positive proportion of $\llambda$. We do not know how to prove this directly, so instead we prove convergence after taking an average over a suitably large neighborhood of $\llambda'$. Intuitively, we show that the convergence holds for all \emph{Lebesgue points} of $\llambda\mapsto W_\infty^{\beta,\llambda}$. We first recall the necessary background from real analysis.
\begin{definition}
A sequence $(E_t)_{t\geq 0}$ of Borel sets is called \emph{nicely shrinking} to $x\in\R^d$ if there exists a positive sequence $(r_t)_{t>0}$ such that $\lim_{t\to\infty}r_t=0$, $E_t\subseteq B(x,r_t)$ and $\liminf_{t\to\infty}\frac{|E_t|}{|B(x,r_t)|}>0$.
\end{definition}

Note that the definition does not require $x\in E_t$ for any $t$.

\begin{thmx}[{\cite[Theorem 7.10]{Rudin}}]\label{thmx:rudin}
Let $f\colon \R^d\to\R$ be Lebesgue integrable. There exists $A\in\mathcal B(\R^d)$ such that $A^c$ has Lebesgue-measure zero and such that, for every $x\in A$ and for every sequence $(E_t)_{t\geq 0}$ nicely shrinking to $x$,
\begin{align*}
\lim_{t\to\infty}\frac 1{|E_t|}\int_{E_t}f(y)\dd y=f(x).
\end{align*}
\end{thmx}

We now prove an almost sure version of Lemma~\ref{lem:mgf_in_prob}.

\begin{lemma}\label{lem:mgf_as}
For $\Pp$-almost all $\omega$ there exists $\Lambda(\omega)\in\mathcal B([0,1]^d)$ such that $|\Lambda(\omega)|=1$ and such that for all $\llambda\in \Lambda(\omega)$ and any bounded $E\in\mathcal B(\R^d)$ with $|E|>0$,
\begin{align}\label{eq:mgf_as}
\lim_{t\to\infty} \frac{1}{|E_t|}\int_{\llambda+E_t}W_t^{\beta,\llambda'}(\omega)\dd \llambda'=W_\infty^{\beta,\llambda}(\omega),
\end{align}
where $E_t\coloneqq E/\sqrt t$.
\end{lemma}

The proof is postponed until after Theorem~\ref{thm:clt_as_pois}. To apply Lemma~\ref{lem:mgf_as}, we use the following easy approximation result whose proof is skipped. We write $x*y$ for the element-wise product $(x_1y_1,\dots,x_dy_d)\in\R^d$ of two vectors $x,y\in\R^d$.

\begin{lemma}\label{lem:weak_conv}
Let $\nnu_\infty$ and $(\nnu_t)_{t\geq 0}$ be probability measures on $\R^d$. For $\eps>0$ and $t\in\R_+\cup\{\infty\}$, let $\nnu_t^\eps$ denote the law of $X*(\1+Y)$ where $X$ and $Y$ are independent with laws $\nnu_t$ and $\operatorname{Unif}([0,\eps]^d)$. If 
\begin{align*}
\nnu_t^\eps\xrightarrow[t\to\infty]w\nnu_\infty^\eps
\end{align*}
for every $\eps>0$, then $\nnu_t\xrightarrow[t\to\infty]w\nnu_\infty$.
\end{lemma}

\begin{proof}[Proof of Theorem~\ref{thm:clt_as_pois}]
Fix $\omega\in\Omegap$ and let $\llambda\in \Lambda(\omega)$, as defined in Lemma~\ref{lem:mgf_as}. Set
\begin{align*}
\nnu_t(\cdot)=\nnu_t^0(\cdot)\coloneqq \mu_{\omega,t}^{\beta,\llambda}\left(\frac{B_t-t\llambda}{\sqrt t}\in\cdot\right)
\end{align*}
and let $\nnu_\infty=\nnu_\infty^0$ be the standard normal distribution on $\R^d$. For $\eps>0$, define $\nnu_t^\eps$ as in Lemma~\ref{lem:weak_conv} and consider the corresponding moment generating function,
\begin{align*}
M_t^\eps(\llambda')\coloneqq &\int_{\R^d} e^{\llambda' \cdot x}\nnu_t^\eps(\dd x)\\
\Big(=&\eps^{-d}\int_{[0,\eps]^d}M_t^0(\llambda'*(\1+\mmu))\dd \mmu\quad\text{for }\eps>0\Big).
\end{align*}
In the second line we have used that the element-wise product $*$ commutes with the scalar product $\cdot$, i.e. $(\boldsymbol x*\boldsymbol y)\cdot\boldsymbol z=\boldsymbol x\cdot(\boldsymbol y*\boldsymbol z)$. We will show that, for any $\eps>0$ and $\llambda'\in\R^d$ with $\llambda_1',\dots,\llambda_d'\neq 0$,
\begin{align}\label{eq:limit}
\lim_{t\to\infty}M_t^\eps(\llambda')=M_\infty^\eps(\llambda').
\end{align}
Since $\{\llambda'\in\R^d:\llambda_1',\dots,\llambda_d'\neq 0\}$ is dense and $M_\infty^\eps$ is continuous, we conclude that $\nnu_t^\eps$ weakly converges to $\nnu^\eps$ for any $\eps>0$ and hence the claim follows by Lemma~\ref{lem:weak_conv}. Now,
\begin{equation}\label{eq:herehereher}
\begin{split}
M_t^\eps(\llambda')
&=\eps^{-d}\int_{[0,\eps]^d} \frac{E^{\llambda}[e^{(\llambda'*(\1+\mmu ))\cdot (B_t-\llambda t)/\sqrt t}e^{\beta H_t(\omega,B)-t(e^{\lambda(\beta)}-1)}]}{W_t^{\beta,\llambda}}\dd \mmu\\
&=\eps^{-d}\int_{[0,\eps]^d} e^{\frac 12\|\llambda'*(1+\mmu)\|_2^2}\frac{W^{\beta,\llambda+\frac{\llambda'*(\1+\mmu)}{\sqrt t}}}{W_t^{\beta,\llambda}}\dd \mmu\\ 
&=\frac{1}{|E_t|}\int_{E_t}e^{\frac 12t\|\mmu\|_2^2 }\frac{W_t^{\beta,\llambda+\mmu}}{W_t^{\beta,\llambda}}\dd \mmu,
\end{split}
\end{equation}
where $E\coloneqq [\llambda_1',(1+\eps)\llambda_1']\times\dots\times[\llambda_d',(1+\eps)\llambda_d']$, $E_t\coloneqq E/\sqrt t$ and where we adopt the convention that  $[a,b]$ is to be interpreted as $[a\wedge b,a\vee b]$. Note that $|E|=\eps^{d}\prod_{i=1}^d|\llambda_i'|>0$. Finally, Lemma~\ref{lem:mgf_as} and a straightforward approximation argument show that
\begin{align*}
\lim_{t\to\infty}M_t^\eps(\llambda') =\frac{1}{|E|}\int_Ee^{\frac 12\|\mmu\|_2^2}\dd \mmu=M_\infty^\eps(\llambda').
\end{align*}
\end{proof}

\begin{proof}[Proof of Lemma~\ref{lem:mgf_as}]
For $\omega\in\Omegap$, $\llambda\in[0,1]^d$ and $T>0$, let
\begin{align*}
S_T(\omega,\llambda)\coloneqq \sup_{s,t\geq T}\left|W_s^{\beta,\llambda}(\omega)-W_t^{\beta,\llambda}(\omega)\right|.
\end{align*}
We have $S_T(\omega,\llambda)\leq 2\sup_{t\geq 0}W_t^{\beta,\llambda}(\omega)$ and, by Fubini's theorem, \eqref{eq:sup_integrable}, and the shift invariance,
\begin{align}\label{eq:more_complicated}
\E\Big[\int_{[0,1]^d}\sup_{t\geq 0} W_t^{\beta,\llambda}\dd\llambda\Big]=\E\left[\sup_{t\geq 0}W_t^{\beta}\right]<\infty.
\end{align}
In particular, for $\Pp$-almost all $\omega$,
\begin{align}\label{eq:in_particular}
\int_{[0,1]^d} S_T(\omega,\llambda)\dd\llambda<\infty.
\end{align}
Moreover, by dominated convergence and shift invariance,
\begin{align}\label{eq:rot}
\lim_{T\to\infty}\E\Big[\int_{[0,1]^d} S_T(\omega,\llambda)\dd\llambda\Big]=\int_{[0,1]^d}\lim_{T\to\infty}\E\big[ S_T(\omega,\llambda)\big]\dd\llambda=\lim_{T\to\infty}\E\big[ S_T(\omega,0)\big]=0
\end{align}
and, since $T\mapsto S_T(\omega,\llambda)$ is non-negative and decreasing, 
\begin{align}\label{eq:almost_all}
\lim_{T\to\infty}S_T(\omega,\llambda)=0
\end{align}
for $\Pp\otimes\operatorname{Leb}|_{[0,1]^d}$-almost all $(\omega,\llambda)$. Using \eqref{eq:in_particular}, we can apply Theorem~\ref{thmx:rudin} to $f(\llambda)=S_{k}(\omega,\llambda)$. Let $A_k(\omega)$ be the set obtained in this way and $A(\omega)\coloneqq \bigcap_{k=1}^\infty A_k(\omega)$. Then, for $\llambda\in A(\omega)$ and $t\geq k$,
\begin{align*}
\left|\frac{1}{|E_t|}\int_{\llambda+E_t}W_t^{\beta,\llambda'}(\omega)\dd\llambda'-W_\infty^{\beta,\llambda}(\omega)\right|&\leq \left|\frac{1}{|E_t|}\int_{\llambda+E_t}W_k^{\beta,\llambda'}(\omega)\dd\llambda'-W_k^{\beta,\llambda}(\omega)\right|\\
&\quad +S_k(\omega,\llambda)+\frac{1}{|E_t|}\int_{\llambda+E_t}S_k(\omega,\llambda')\dd\llambda'\\
&\xrightarrow{t\to\infty} 2S_k(\omega,\llambda)\\
&\xrightarrow{k\to\infty}0,
\end{align*}
where we used the continuity of $\llambda'\mapsto W_k^{\beta,\llambda'}(\omega)$ for fixed $k$ in the first limit and \eqref{eq:almost_all} in the final line.
\end{proof}

\begin{proof}[Proof of Theorem~\ref{thm:clt_as_bond}]
The proof is similar to Proposition~\ref{prop:defect} with $[0,1]^d$ replaced by $[-\lambda_0,\lambda_0]^d$ and where the moment generating function of the simple random walk has to be used in place of that of Brownian motion in \eqref{eq:herehereher}. We explain how to prove a result corresponding to Lemma~\ref{lem:mgf_as} without relying on shift invariance. Namely, we obtain an analog of \eqref{eq:rot} from \eqref{eq:W_nW_m} and for \eqref{eq:more_complicated} we compute, using Proposition~\ref{prop:drift}(i) and Doob's maximal inequality,
\begin{align*}
\E\Big[\int_{[-\lambda_0,\lambda_0]^d}\sup_{n\in\N} W_n^{\beta,\aalpha(\llambda)}\dd\llambda\Big]
&\leq C\int_{[-\lambda_0,\lambda_0]^d}\sup_N\|W_N^{\beta,\aalpha(\llambda)}\|_p\,\dd\llambda<\infty.\tag*{\qedhere}
\end{align*}
\end{proof}

\section*{Acknowledgements}
The author is grateful to Ryoki Fukushima for many inspiring discussions and helpful suggestions. This work was supported by the JSPS Postdoctoral Fellowship for Research in Japan, Grant-in-Aid for JSPS Fellows 19F19814.
\bibliographystyle{plain}
\bibliography{note}

\end{document}